\documentclass[preprint,12pt]{elsarticle}
\usepackage{amsfonts, amsmath,latexsym,graphics,amsthm,amssymb} 
\usepackage{color}
\newtheorem{theo}{Theorem}[section]
\newtheorem{lem}[theo]{Lemma}

\newtheorem{defi}{Definition}[section]
\newtheorem{rem}{Remark}[section]

\def\Div{\operatorname{div}}
\def\sgn{\operatorname{sgn}}
\def\R{\mathbb{R}}

\journal{Applied Numerical Mathematics (accepted)}
 
\begin{document}

\begin{frontmatter}

\title{A Semi-Lagrangian Scheme for the Game $p$-Laplacian via  $p$-averaging}

%\author[rm,rm,lc,lc]{M. Falcone, S. Finzi Vita, T. Giorgi and R.G. Smits}

\author[rm]{M. Falcone\corref{cor1}}
\ead{falcone@mat.uniroma1.it}

\author[rm]{S. Finzi Vita\corref{cor1}}
%\author[rm]{S. Finzi Vita}
\ead{finzi@mat.uniroma1.it}

\author[lc]{T. Giorgi\corref{cor2}}
\ead{tgiorgi@math.nmsu.edu}

\author[lc]{R.G. Smits}
\ead{rsmits@math.nmsu.edu}

%\ead{falcone@mat.uniroma1.it,  finzi@mat.uniroma1.it, tgiorgi@NMSU.Edu,rsmits@nmsu.edu}

% \corref{cor1} \corref{cor2}
%\cortext[cor1]{Corresponding author. First version submitted: June 26, 2011. Revised version submitted:  July 31, 2012. Second revised version submitted:  November 15, 2012.}
\cortext[cor1]{ Funding to these authors has been provided by PRIN 2009 ``Modelli numerici per il calcolo scientifico ed applicazioni avanzate''.}
\cortext[cor2]{ Funding to this author was provided by the National Science
Foundation Grant \#DMS-1108992.}

%\cortext[cor1]{Dipartimento di Matematica, Sapienza Universit\`a di Roma, P.le Aldo Moro, 2 - 00185 Roma, Italy}
%\cortext[cor2]{Department of Mathematical Sciences, New Mexico State University, Las Cruces, NM 88003-8001, USA}

\address[rm]{Dipartimento di Matematica, Sapienza Universit\`a di Roma\\ P.le Aldo Moro, 2 - 00185 Roma, Italy}
\address[lc]{Department of Mathematical Sciences, New Mexico State University\\ Las Cruces, NM 88003-8001, USA}
%\date{}
%\maketitle

%\vspace{20pt}

%%%%%%%%%%%%%%%%%%%%%%%%%%%%%%%%%%%%%%%%%%%%%%%%%%%%%%%%%%%%%%%%%%%%%%%%%%

\begin{abstract}
We present and analyze an approximation scheme for the two-dimensional game $p$-Laplacian in the framework of viscosity solutions. The approximation is based on  a semi-Lagrangian scheme  which exploits the idea of $p$-averages. We study the properties of the scheme and prove that  it converges, in particular cases, to the  viscosity solution of the game $p$-Laplacian. We also present a numerical implementation of the scheme for different values of $p$; the numerical tests show that the scheme is accurate. 
\end{abstract}


%\noindent{\small{\bf{AMS Classification:}}} 
%{\small{65M06, 65M12, 49L25}.}\hfill\break
% \noindent{\small{\bf{Keywords:}}} {\small{$p$--Laplacian, Tug of war game, Hamilton--Jacobi equations, semi-Lagrangian scheme, convergence, viscosity solutions.}}\hfill
 %\break
 
 \end{frontmatter}

\section{Introduction}

 The game $p$-Laplace operator has been recently introduced in \cite{PS} to model a stochastic game called {\it tug-of-war with noise}. Part of the interest for this class of operators arises from the fact that it includes, as particular cases, the operator in the Aronsson equation \cite{BEJ08}, the infinity Laplacian \cite{PSSW}, the motion by mean curvature operator \cite{KS}, and for $p=2$, a multiple of  the  ordinary Laplacian. For the connections between these operators and differential games see also \cite{E07}.
 
 The equation associated to the game $p$-Laplacian has the same solutions as the variational $p$-Laplacian only in the homogeneous case, and has the advantage in the non-homogeneous case of being a combination of other $p$-Laplacians. In particular, we would like to stress the fact that even for the non-homogeneous case the game $\infty$-Laplacian is the limit as $p\to\infty$ of the game $p$-Laplacian.
 
 Our work strongly relies on  the general philosophy illustrated in the paper of Peres and Sheffield \cite{PS}, which indicates that to study $p$-harmonic functions one can look at discrete versions of the $p$-operators; these correspond to stochastic processes with paths that are nonlinearly averaged, ranging from motion by mean curvature to Brownian motion to diffusions generated by the Arronsson operator. It is important to note that our work is in the framework of weak solutions in the viscosity sense (see \cite{E10} for an introduction and \cite{CIL92} for a guide to viscosity solutions for  second order problems). However, the difference between the analytic case of Peres and Sheffield \cite{PS}, and our approximation scheme is in the fact that we need the value at a fixed point to depend only on a discrete number of values in space. In this respect, our construction starts from an approximation of $p$-averages in order to keep a strong link between the continuous and the discrete operator. Specifically, one of our goals is to prove that our scheme is consistent with the continuous operator proposed by Peres and Sheffield.

Semi-Lagrangian schemes for nonlinear Hamilton-Jacobi equations have been studied and analyzed  by several authors. The starting point is the discrete version of the characteristic method which leads to a discrete Lax-Hopf formula for first order Hamilton-Jacobi equation \cite{FF02}. It is interesting to note that semi-Lagrangian schemes allow for large time steps still satisfying stability conditions. A comprehensive introduction to semi-Lagrangian methods for linear and first order Hamilton-Jacobi equations is contained in the book by Falcone and Ferretti  \cite{FF11}.  For second order problems some results have been obtained for stationary and evolutive equations related to stochastic control problems in \cite{CF95}, whereas a presentation of the treatment of second order terms in SL schemes has been given in \cite{Fe10}. Recent extensions to mean curvature driven flows have been analyzed in \cite{CFF10}.

In the homogeneous case the game $p$-Laplacian coincides with the variational $p$-Laplacian for which several approximation methods have been proposed. Some of these schemes are based on finite elements and show convergence, also establishing a priori error estimates, see e.g. the paper by Barrett and Liu  \cite{BL94}. However, finite elements are not the most popular techniques for nonlinear degenerate equation. Finite difference approximation schemes for degenerate second order equations have been proposed and analyzed by Crandall and Lions in \cite{CL96}, and in several papers by Oberman \cite{O04,O104,O06}.  More recently,  finite volumes schemes have been presented by Andreyanov, Boyer and Hubert in \cite{ABH06}, for the variational $p$-Laplacian.

In the non-homogeneous case the game $p$-Laplacian interprets the non-homogeneity, $f$, as a multiple of a running payoff for one of the players in a two-player, zero-sum game while for the variational $p$-Laplacian the non-homogeneity is interpreted as a potential. From the numerical point of view, the variational $p$-Laplacian is typically studied with homogeneous Dirichlet boundary conditions. In this case, one has  Poincar\'e inequalities when the region and solutions are smooth, which help in the proofs of convergence and error estimates. The game $p$-Laplacian has fewer tools, relying essentially on monotonicity properties and on the notion of viscosity solutions. For this reason we believe that our results will be a useful contribution to the theory.

  The paper is organized as follows. In Section~2 we introduce the formal definition of the game $p$-Laplacian, and give a precise definition of  viscosity solutions for our problem, as well as a brief description of the stochastic game {\it tug-of-war with noise}.  We formulate  our approximation scheme in Section~3, where we also give insights into its construction. The analysis of the properties of the scheme, and a proof of convergence for the homogeneous case when $p\geq 2$ are presented in Section~4. In Section~5 we discuss in detail  the numerical implementation of the scheme on a rectangular grid, and we conclude with some numerical tests in Section~6. For the sake of completeness, we also add in Section~7 a technical appendix on some elementary properties of the $p$-average of finite sets of real numbers.

\section{Game $p$-Laplacian}\label{sec-defin}

The $p$-Laplace operator, which we refer to as {\it the variational} $p$-Laplacian, for $1\leq p<\infty$, is defined by
\begin{equation}\label{p-var} 
\Delta_p \, u :=  \Div \left(|\nabla u|^{p-2} \, \nabla u \right),
\end{equation}
whereas for $p=\infty$,  traditionally  is given by
 \begin{equation*}%\label{def-infinity}
 \Delta_\infty u := \displaystyle{  \sum_{i,j} \, \frac{\partial u}{\partial x_i} \,\frac{\partial u}{\partial x_j} \, \frac{\partial^2 u}{\partial x_i \partial x_j}}.
 \end{equation*}

The subject of our numerical study is  the Dirichlet boundary value problem for the so-called game $p$-Laplacian introduced by Peres and Sheffield \cite{PS},  which for $1<p<\infty$ reads as follows:

\begin{equation}\label{p-game}
\left\{\begin{array}{lll}
&-\Delta_p^G u = f\quad& \text{ in } \Omega, \\
& u =\,F\quad &\text{ on } \partial \Omega,
\end{array} \right.
\end{equation}
where
\begin{equation}\label{def-p-game}
\Delta_p^G u := \displaystyle{ \frac 1p\, |\nabla u|^{2-p} \, \Div \left(|\nabla u|^{p-2} \, \nabla u\right)}.
\end{equation}

We require $f$ and $F$ to be continuous in their domain of definition. Additionally, we assume $f$ either identically equal to zero or never zero. In the sequel,  we will then always consider the cases $f\equiv 0$ or $f>0$ ( without loss of generality). Here $\Omega\subset \R^n$ is a bounded smooth domain.

Note that for $p=2$, one has $\Delta_2^G  = \frac 12 \Delta_2$, that is one-half of the Laplacian, which is the infinitesimal generator for a Brownian motion.
 
If $u$ is a smooth function, by expanding the derivatives,  we obtain 
\begin{equation}\label{e2-1}
\Delta_p^G u = \frac 1p \, \Delta_2 \, u + \frac{p-2}p \, |\nabla u|^{-2} \, \sum_{i,j} \, \frac{\partial u}{\partial x_i} \,\frac{\partial u}{\partial x_j} \, \frac{\partial^2 u}{\partial x_i \partial x_j},
\end{equation}
 therefore, by taking the limit for $p\to \infty$, one is naturally lead to the following definition of the game $\infty-$Laplacian:
 \begin{equation}\label{def-infinity-game}
 \Delta_\infty^G u := \displaystyle{ |\nabla u|^{-2} \, \sum_{i,j} \, \frac{\partial u}{\partial x_i} \,\frac{\partial u}{\partial x_j} \, \frac{\partial^2 u}{\partial x_i \partial x_j}}.
 \end{equation}
The game $1$-Laplacian is defined in terms of the Laplacian and the game $\infty-$Laplacian:
\begin{equation}\label{1-game}
\Delta_1^G u :=  \Delta_2 u- \Delta_\infty^G u.
\end{equation}

We would like to point out  that with this notation, the expansion in $(\ref{e2-1})$ allows us to think of $\Delta_p^G$ as the convex combination of the two limiting cases, that is
\begin{equation}\label{e2-2}
\Delta_p^G= \frac 1p \, \Delta_1^G  \, + \,  \frac 1q  \, \Delta_\infty^G,
\end{equation}
with $q$ the conjugate exponent  of $p$ (i.e. $\frac1p+\frac1q=1$).

At the points where $\nabla u \neq 0$, the game 1-Laplacian and the game $\infty$-Laplacian can be thought as the second derivative in the orthogonal direction of $\nabla u$ and in the direction of $\nabla u$, respectively.
That is,
 \begin{equation}\label{1-game-der} 
\Delta_1^G \, u = |\nabla u^\perp|^{-2}  <D^2 u  \, \nabla u^\perp, \nabla u^\perp>,
\end{equation}
and
 \begin{equation}\label{infty-game-der} 
\Delta_\infty^G \, u = |\nabla u|^{-2} <D^2 u \, \nabla u,\nabla u>,
\end{equation}
where $D^2 u$ denotes the Hessian matrix.

The variational $p$-Laplacian can be obtained as the Euler-Lagrange equation of an energy functional,  a fact that does not hold for  the game $p$-Laplacian. Additionally, while the variational $p$-Laplacian is degenerate elliptic for $2<p<\infty$ and singular for $1\leq p <2$,  the game $p$-Laplacian is singular for every $p\neq 2$, so  suitable definitions of viscosity  solutions are needed.

Juutinen  and al. \cite{JLM} have shown that for the variational $p$-Laplacian, when $1<p<\infty$, the notions of viscosity solution and weak solution are equivalent. The interested reader can find in the survey  \cite{CIL92} a number of results on viscosity solutions for second order problems. Note that, in the homogeneous case, i.e. when $f \equiv 0$, the solutions of the two operators agree with each other.

Various definitions of viscosity solutions for the game $p$-Laplacian can be given and are found in the literature. The most suitable for our treatment is the one obtained by following the definition in the  classical paper of Barles and Souganidis \cite{BS91}.

In what follows, we will restrict ourselves to the two-dimensional case, that is we will take $\Omega \subset \R^2$.

 \begin{defi}\label{visco-sol-BS}
Consider a smooth domain $\Omega \subset{\R}^2$, and let $1<p\leq \infty$, $q$ such that $1/p+1/q=1$. If $f$ is a continuous function, we say that an upper semi-continuous function  [respectively, lower semi-continuous] $u:\Omega\rightarrow \R$  is a viscosity subsolution [supersolution] of 
\begin{equation}\label{p-equation}
-\Delta_p^G u (x) = f(x) \mbox{ in } \Omega,
\end{equation}
if for any $\phi \in C^2(\Omega)$ such that $u-\phi$ has a local maximum [local minimum] at $x\in\Omega$, we have
\begin{itemize}
\item[(i)]  $\displaystyle{-\Delta_p^G \phi(x)  \leq f(x)} \quad [\displaystyle{-\Delta_p^G \phi(x)  \geq f(x)}] \quad \mbox{ if } \nabla \phi(x) \neq 0$\ ;
\item[(ii)] if $\lambda_1\leq \lambda_2$ denote the eigenvalues of $D^2\phi(x)$, then:

$\displaystyle{-\frac {\lambda_1}p - \frac {\lambda_2}q \leq f(x)} \quad [\displaystyle{- \frac {\lambda_1}q - \frac {\lambda_2}p\geq f(x)}] \quad   \mbox{if }  \nabla \phi(x) = 0 \mbox{ and }\ p\geq 2$;

$\displaystyle{-\frac {\lambda_1}q - \frac {\lambda_2}p \leq f(x)}  \quad  [\displaystyle{ - \frac {\lambda_1}p - \frac {\lambda_2}q\geq f(x)}]\quad \mbox{if }  \nabla \phi(x) = 0  \mbox{ and }1<p< 2$.
\end{itemize}
\end{defi}
 
 \  \
 
\begin{rem}\label{visc-sol-rem}
Part (ii) of the definition of viscosity subsolution [supersolution] is implied by the condition:
\begin{itemize}
\item[(ii)'] $\displaystyle{-\Delta_2^G \phi(x) \leq f(x)}\quad [\displaystyle{-\Delta_2^G \phi(x) \geq f(x)}] \quad \mbox{whenever } \, \nabla \phi(x) = 0$.
\end{itemize}

This is a consequence of the fact that 
\begin{equation*}%\label{visc-impl}
-  \frac {\lambda_1} p - \frac {\lambda_2}q  \leq  -\Delta_2^G   \phi(x) \leq -\frac  {\lambda_1}q - \frac {\lambda_2}p,\ \mbox{ if } \, p\geq 2,
\end{equation*}
and
\begin{equation*}%\label{visc-impl-2}
-  \frac { \lambda_1}q - \frac {\lambda_2}p  \leq  -\Delta_2^G   \phi(x) \leq -\frac {\lambda_1}p - \frac {\lambda_2}q,\ \mbox{ if } \, 1<p<2.
\end{equation*}
\end{rem} 
 
Uniqueness for viscosity solutions of nonlinear operators that are singular at isolated points,  typically does not depend on the particular value one assigns to these points as long as this is chosen in a consistent manner (see for example Section~9 in \cite{CIL92}), additionally our numerical results show numerical convergence to solutions that verify $(ii)'$. Therefore, we will use the following definition for viscosity solution of the game $p$-Laplacian:

 \begin{defi}\label{G-visco-sol}
A function $u$ is a {\em viscosity solution} of (\ref{p-equation}), for $1<p\leq\infty$ if $u$ is a subsolution and a supersolution according to $(i)$ of Definition~\ref{visco-sol-BS} and $(ii)'$ in Remark~\ref{visc-sol-rem}.
\end{defi}

As we said in the introduction, the main ingredient of our approximation is the way we discretize using $p$-averages.
Let us recall the notion of $p$-average of a set of numbers.

 \begin{defi}\label{num-p-ave}
 Given a finite set of real numbers, $S=\{s_1, s_2, .....,s_m\}$, we denote by
$A_p(S)$ the {\em $p$-average} of its elements, that is $A_p(S)$ is such that
\begin{equation}\label{p-ave}
\displaystyle{\sum_{j=1}^m \, \left|s_j-A_p(S)\right|^p =  \min_{c \in {\R}} \sum_{j=1}^m \, |s_j-c|^p} \quad \text{ if } \quad  1< p <\infty,
\end{equation}
\begin{equation}\label{infty-ave}
\displaystyle{A_\infty(S)= \frac12 \, \left[ {\displaystyle{\max_{s_j \in S} s_j +\min_{s_j \in S} s_j}}\right],}
\end{equation}
and
\begin{equation}\label{one-ave}
\displaystyle{A_1(S)=  \text{ median of } S.}
\end{equation}
\end{defi}
Since the median of an even number of points is not uniquely defined, in (\ref{one-ave}) we follow tradition and take it to be the average of the two middle points.

\smallskip Note that by convexity $A_p(S)$ above is unique for $1<p\le \infty$. For $p=2$, $A_2(S)$ is the arithmetic mean of the numbers in the set $S$:
$$
A_2(S)=\frac 1m{\displaystyle{\,\sum_{j=1}^m s_j\,}}.
$$

We end this section with a brief description of the two-player, zero-sum game called {\em tug-of-war with noise}. In this game, we fix a parameter $\epsilon>0$, a $1<p<\infty$, a domain $\Omega$, a continuous running cost $h:\Omega\to \R$ with $h\geq 0$, and a continuous exit cost $F:\partial\Omega\to \R$, as well as an initial position $x=x_0 \in \Omega$. A token is placed at $x_0$, and at each stage, $k$, a fair coin is flipped. The winner of the coin flip picks any direction vector, $v_k$, with $|v_k|\leq\epsilon$, to which a random noise vector $z_k$ is added. The vector $z_k$ is equally likely to be one of the two vectors orthogonal to $v_k$ of length $\displaystyle{\sqrt{\frac 1{p-1}} \, |v_k|}$. Then the token is moved to $x_k=x_{k-1}+v_k+z_k$ and the play continues until the token is within a distance $\displaystyle{\widehat\epsilon=\left(1+\sqrt{\frac 1{p-1}}\right) \epsilon}\,$ from $\partial\Omega$. In that case, the winner picks $x_k \in \partial \Omega$ with $|x_k-x_{k-1}|\leq \widehat\epsilon$. The payoff to one of the players, say player I, from the other player, say player II, is $\displaystyle{F(x_k)+ \epsilon^2 \sum_{i=0}^{k-1}h(x_i)}$. Under various conditions, Peres and Sheffield \cite{PS} show that when both players choose optimal strategies, as $\epsilon \to 0$ there is an expected value $u(x)=u(x_0)$ which solves the boundary value problem for the game $p$-Laplacian equation given in  (\ref{p-game}) with $\displaystyle{f=\frac 2q h}$. Note that as $p\to\infty$ the noise vector disappears, while for $p=2$ it has the same length as the chosen direction vector, resulting in a two-dimensional random walk.

%%%%%%%%%%%%%%%%%%%%%%%%%%%%%%%%%%%%%%%%%%%%%%%%%%%%%%%%%%%%%%%%%
%%%%%%%%%%%%%%   CONSTRUCTION OF THE SCHEME
%%%%%%%%%%%%%%%%%%%%%%%%%%%%%%%%%%%%%%%%%%%%%%%%%%%%%%%%%%%%%%%%%
\section{Construction of the approximation scheme}

We arrive at our approximation scheme inspired by the work of Peres and Sheffield \cite{PS}, where the game interpretation of the game $p$-Laplace operator is based on averaging over non-Markovian paths, by asking the question if the $p-$game operators have an averaging characteristic, and if this can be captured by some quantity. Having this in mind, looking at  the numerical approximations of Oberman \cite{O04,O104} of the $1$-Laplacian and the $\infty$-Laplacian, as well as at the standard central difference approximation of the $2$-Laplacian, we notice that they all can be rewritten in terms of their corresponding $p$-average. 

From these observations, we arrive at the conclusion that there should be an inherent averaging characteristic in all the $p$-operators, and that the notion of $p$-average is a possible candidate for the correct quantity describing it. We refer to the work \cite{GS}  for an analytical result on averaging properties of the $p$-Laplacian in terms of a continuous $p$-average (see also \cite{MPR}). 

To stress how the notion of $p$-average comes into play quite naturally, when dealing with approximation schemes for the game $p$-Laplacian,  let's look at some known approximation schemes and rethink them in terms of $p$-averages. To avoid cumbersome notations we present the approximation scheme in $\R^2$ but the extension to the general multidimensional case follows along the same lines.

For $p=2$, using standard central differences, at a point ${\bf x}=(x_1,x_2)$, for $h \in {\R}$ small,  we have
\begin{eqnarray}
&&\Delta_2^G \, u(x_1,x_2) \approx  \frac 1{2 \, h^2} \left[u(x_1+h,x_2)+u(x_1,x_2+h)  \right. \nonumber \\
&& \qquad  \left.+u(x_1-h,x_2)+u(x_1,x_2-h)-4 \, u(x_1,x_2)\right],\nonumber
\end{eqnarray}
which, reordering terms in a suitable way, gives
\begin{equation}
\Delta_2^G  \, u({\bf x}) \approx \frac{2}{ \, h^2} \left[ A_2(C_h( {\bf x},u)) - \, u({\bf x})\right]. \nonumber
\end{equation}
In general, we will denote by  $C_h( {\bf x},u)$ the 
set of values used to compute the approximation at a point ${\bf x}$,  in this case 
$
C_h( {\bf x},u) =\{ u(x_1+h,x_2), u(x_1,x_2+h), u(x_1-h,x_2),u(x_1,x_2-h)\}. 
$

\ \

For $p=\infty$, the scheme in \cite{O04} can be rewritten as
\begin{equation}
\Delta_\infty^G \, u({\bf x}) \approx  \frac{2}{ \, h^2} \left[ A_\infty(C_h( {\bf x},u)) - \, u({\bf x})\right], \nonumber
\end{equation}
where now  $C_h( {\bf x},u)$ is a discrete set of values of $u$ on the sphere of radius $h$ centered at ${\bf x}$, and  the distribution and number of points on the sphere influences the accuracy of the approximation in a fundamental way. 

It is relevant to mention that also for $p=2$ one could pick as $C_h( {\bf x},u)$ a larger set of values of $u$ on the sphere, but since the Laplacian is a linear operator this would not increase the accuracy.  

A similar scheme that uses the $1-$average  can be constructed in the case $p=1$, in view of the interpretation as second directional derivative given by $(\ref{1-game-der})$, see \cite{O104} for the parabolic case.

The generalization to the game $p$-Laplacian of these interpretations using averages suggests the following approximation:
\begin{equation}\label{approxGp}
\Delta_p^G \, u({\bf x}) \approx   \frac{2}{ \, h^2} \left[ A_p(C_h( {\bf x},u)) - \, u({\bf x})\right],
\end{equation}
where again $C_h( {\bf x},u)$ would be a suitable discrete set of values of $u$ on the sphere of radius $h$ centered at ${\bf x}$.

We are then lead to the following approximation scheme for the Dirichlet problem (\ref{p-game}): 
\begin{equation}\label{ellipticscheme}
\hspace{-0.1cm} S(\rho,{\bf x}, u({\bf x}), u) = 0 \, \mbox{ in }  \overline \Omega,
\end{equation}
where the positive discretization parameters are represented by the vector $\rho:=(h,\Delta \theta)$ (with $h$ the spatial step and $\Delta\theta$ the angular resolution), and  $S: [0,1) \times (0,\pi/2] \times \overline{\Omega}\times {\R} \times L^\infty(\overline{\Omega}) \longrightarrow {\R}$ is defined as
\begin{equation}\label{schemeGp}
S(\rho,{\bf x}, u({\bf x}), u) =
\left\{\begin{array}{ll}
\hspace{-0,1cm}-\frac{2}{ \alpha^2 \, h^2} \left[ A_p(C_h^{\Delta \theta}( {\bf x},u;\alpha)) - \, u({\bf x})\right]- f({\bf x}) &\text{in }\Omega, \\ \\
 u({\bf x}) - F({\bf x})& \hspace{-0.3cm}\text{on }\partial\Omega.
\end{array} \right.
\end{equation}
Here, if $d_\Omega<\infty$ denotes the diameter of $\Omega$, $\alpha=\alpha({\bf x})$ is a dilation parameter such that $0<\alpha({\bf x})\le dist({\bf x}, \partial{\Omega})<d_\Omega$. Finally,  $C_h^{\Delta \theta}( {\bf x},u;\alpha)$ is now a suitably chosen discrete set of values of $u$, taken on the sphere of center  ${ \bf x}$ and radius $h \,\alpha$, associated to the angular resolution $\Delta \theta$. 

There is some freedom in how to choose the points in $C_h^{\Delta \theta}( {\bf x},u;\alpha)$, but as in \cite{FT09} we follow a standard discretization of the sphere of radius $h \,\alpha$ centered at ${\bf x}$, and take ${\bf y}_i={\bf x}+ h \,\alpha\, {\bf r}_i$ (with $ \Delta\theta=\pi/(2\,m)$ and ${\bf r}_i = \left(\cos i\Delta\theta, \sin i \Delta\theta\right)$), so that
\begin{equation}\label{stencil1}
C_h^{\Delta \theta}( {\bf x}, u;\alpha )=\{u({\bf y}_i), \, i=0,...,4m-1\}.
\end{equation}
Note that with this choice of $m$, if a direction ${\bf r}_i$ is in the set of admissible directions so is its opposite,  $-{\bf r}_i$, as well as its orthogonal and its reflections with respect to each of the axes. Also, note that with our choice of $\alpha$ we have that ${\bf y}_i \in \overline \Omega$ for every $i$, so our set $C_h^{\Delta \theta}( {\bf x}, u;\alpha )$ is well-defined.

%%%%%%%%%%%%%%%%%%%%%%%%%%%%%%%%%%%%%%%%%%%%%%%%%%%%%%%%%%%%%%%%%
%%%%%%%%%%%%%%   CONVERGENCE
%%%%%%%%%%%%%%%%%%%%%%%%%%%%%%%%%%%%%%%%%%%%%%%%%%%%%%%%%%%%%%%%%
 
\section{Study of convergence}
Let us analyze the convergence of our approximation scheme using the framework provided by the classical result of convergence for fully nonlinear second order elliptic equations of  Barles and Souganidis presented in \cite{BS91}. Assuming that a comparison principle is available for the exact solution, in this approach convergence to viscosity solutions is implied by the monotonicity, stability and consistency of the scheme.

We prove monotonicity for the general case, consistency for $p\geq 2$, and stability for the case $f=0$. Therefore, we have formal convergence of the scheme for the case $p\geq 2$ and $f=0$. Nevertheless, the numerical experiments we run, and which we illustrate in Section \ref{numtest}, show convergence in the general case.

We follow \cite{BS91} and define $G_p: T^2 \times {\R}^2 \times {\R} \times {\bar \Omega} \to {\R}$, where $T^2$ is the set of $2\times 2$ real symmetric matrices, by first introducing the function $H_p:T^2 \times {\R}^2 \times {\Omega} \to {\R}$:
\begin{equation}
H_p(M, {\bf r}, {\bf x}) = \left \{ \begin{array}{ll}
\displaystyle{-\frac 1p  <M \frac{{\bf r^\perp}}{|{\bf r^\perp}|}, \frac{{\bf r^\perp}}{|{\bf r^\perp}|}> - \frac 1q <M \frac{{\bf r}}{|{\bf r}|}, \frac{{\bf r}}{|{\bf r}|}>  - f({\bf x})}  \, &\mbox{ if }  {\bf r} \neq {\bf 0}\\ \\
-\frac 12 {\mbox tr}(M) - f({\bf x})   \,&\mbox{ if } {\bf r} = {\bf 0}, 
\end{array} \right. \nonumber
\end{equation}
%\begin{eqnarray}
%&&H_p(M, {\bf r}, u, {\bf x}) \nonumber \\ 
%&&= \left \{ \begin{array}{ll}
%\displaystyle{-\frac 1p  <M \frac{{\bf r}}{|{\bf r}|}, \frac{{\bf r}}{|{\bf r}|}> - \frac 1q <M \frac{{\bf r}}{|{\bf r}|}, \frac{{\bf r}}{|{\bf r}|}>  - f({\bf x})}  \, & \mbox{ if }  {\bf r} \neq {\bf 0}\\ \\
%- \frac 12 {\mbox tr} M - f({\bf x})   \,& \mbox{ if } {\bf r} = {\bf 0}, 
%\end{array} \right. \nonumber
%\end{eqnarray}
and then setting
\begin{equation}\label{Fp}
G_p(M, {\bf r}, u,  {\bf x})= \left \{ \begin{array}{ll}
H_p(M, {\bf r}, {\bf x})  & \mbox{ for } {\bf x} \in \Omega,  \\ \\
u({\bf x}) - F({\bf x})   \,& \mbox{ for } {\bf x} \in \partial \Omega.
\end{array} \right. 
\end{equation}

In this notation, the Dirichlet problem (\ref{p-game}) is expressed as
\begin{equation}\label{Fp-game}
G_p(D^2u,Du,u,{\bf x})=0 \quad \mbox{ for } \, {\bf x}\in\overline\Omega.
\end{equation}

\begin{rem} A viscosity solution of (\ref{Fp-game}) is a function $u$ that verifies  Definition~\ref{G-visco-sol} in $\Omega$ and satisfies the boundary conditions
\begin{eqnarray}
&\max(H_p(D^2u,Du,{\bf x}), u-F) \geq 0\,  \mbox{ on } \partial \Omega, \nonumber \\
&\min(H_p(D^2u,Du,{\bf x}), u-F)  \leq 0\,  \mbox{ on } \partial \Omega. \nonumber
\end{eqnarray}
\end{rem}

The ellipticity of $G_p$ is a trivial consequence of its definition:
\begin{lem}
$G_p$ is elliptic, namely  for all $({\bf r},u,{\bf x}) \in  {\R}^2 \times {\R} \times {\bar \Omega}$ we have
$$
G_p(M, {\bf r}, u, {\bf x}) \leq G_p(N, {\bf r}, u, {\bf x}),
$$
for all $M,N \in T^2$ such that $M-N$ is positive semidefinite.
\end{lem}

Monotonicity of the scheme stems from the monotonicity property of the $p$-average, which we derive, for the convenience of the reader, in Lemma~\ref{p-mono} of the Appendix.
\begin{theo}\label{monotonicity}
Let $u, v \in L^\infty(\overline \Omega)$, if $u(\bf{x}) \ge$ $v(\bf{x})$ in $\overline\Omega$ then for all $p\ge 1$, $\rho \in [0,1)\times(0,\pi/2]$, ${\bf x}\in \overline\Omega$, and $t\in{\R}$ it holds
$$
S(\rho,{\bf x}, t, u)  \leq S(\rho,{\bf x}, t, v). 
$$
\end{theo}
\begin{proof}
If ${\bf x}\in\partial\Omega$ then $S(\rho,{\bf x}, t, u) = t -F({\bf x}) = S(\rho,{\bf x}, t, v)$ and the claim is clearly true. 
If ${\bf x}\in\Omega$ we have that 
\begin{eqnarray*}
S(\rho,{\bf x}, t, u)=- \frac{2}{ \alpha^2 \, h^2} \left[ A_p(C_h^{\Delta \theta}( {\bf x},u;\alpha)) - \, t\right]- f({\bf x}) \\
\leq - \frac{2}{ \alpha^2 \, h^2} \left[ A_p(C_h^{\Delta \theta}( {\bf x},v;\alpha)) - \, t\right] -f({\bf x}) = S(\rho,{\bf x}, t, v),
\end{eqnarray*}
since $A_p(C_h^{\Delta \theta}( {\bf x},u;\alpha)) \geq A_p(C_h^{\Delta \theta}( {\bf x},v;\alpha)) $, thanks to the assumption
$u\geq v$ in $ \overline \Omega$ and Lemma~\ref{p-mono}. 
\end{proof}

To prove consistency, we start by showing that in the case $p\geq 2$ our approximation has the correct behavior in the internal points of the domain.
\begin{theo}\label{consist}
Let  $p\geq2$. For all   ${\bf x}\in \Omega$ and $\phi\in C^\infty(\overline \Omega)$, we have that 
\begin{eqnarray}
\lim_{(h,\Delta \theta)\to 0}\ \frac 2{\alpha^2\,h^2} \left[A_p(C_h^{\Delta \theta}( {\bf x}, \phi; \alpha)) - \phi({\bf x}) \right]= \left \{ \begin{array}{ll}
 \Delta^G_p \phi({\bf x}) & \mbox{ if }\,  \nabla  \phi({\bf x})  \neq 0, \\ \\
 \Delta^G_2 \phi({\bf x}) & \mbox{ if }\,  \nabla  \phi({\bf x})  = 0.
\end{array} \right. 
\end{eqnarray}
\end{theo}

\begin{proof} Assume $\nabla \phi({\bf x})  \neq 0 $, and denote by ${\bf e}_1=(1,0)$; without loss of generality, we can assume ${\bf x}=(0,0)$ and $\nabla \phi ({\bf x}) = |\nabla \phi({\bf 0})| \, {\bf e}_1$.  Equation $(\ref{infty-game-der})$ then gives 
\begin{equation}\label{d11}
\Delta^G_\infty \phi ({\bf 0}) = \partial_{11} \phi({\bf 0}),
\end{equation}
while $(\ref{1-game-der})$ yields
\begin{equation}\label{d22}
\Delta^G_1 \phi ({\bf 0}) = \partial_{22} \phi({\bf 0}).
\end{equation}

The $M:= 4 \,m$ points in the set $C_h^{\Delta \theta}( {\bf x}, \phi;\alpha )$ are now on a sphere of center ${\bf 0}$ and radius $\alpha h$, with $0<\alpha<d_\Omega$,  see (\ref{stencil1}),   that is they are given by 
\begin{equation}\label{stencil0}
\xi_j=\alpha h\, (\cos \theta_j, \sin \theta_j) \, \mbox{ for } j=1..M, 
\end{equation}
where the $\theta_j$ are uniformly distributed angles verifying $|\theta_{j+1}-\theta_j|=\Delta \theta$.  We use Taylor's expansion to obtain  
 $$
 \phi(\xi_j)=\phi({\bf 0}) +\nabla \phi ({\bf 0}) \cdot \xi_j + \frac 12 <D^2 \phi({\bf 0}) \xi_j, \xi_j> + o(\alpha^2 \, h^2),
 $$ 
and by Lemmas~\ref{p-stab} and~\ref{trivial} it is enough to show that the $p$-average of  $\displaystyle{\frac 2{\alpha^2\,h^2}\left(\nabla \phi ({\bf 0}) \cdot \xi_j + \frac 12 <D^2 \phi({\bf 0}) \xi_j, \xi_j>\right)}$ tends to $\Delta^G_p \phi({\bf 0})$ as $h$ and $\Delta \theta$ tend to 0. 

If $p=\infty$, consistency is proven as in Oberman \cite{O04}. If $p<\infty$, we employ  (\ref{d11}), (\ref{d22}), (\ref{e2-2}) and the definition of $\xi_j$ to rewrite our elements:
\begin{eqnarray*}
&& \nabla \phi ({\bf 0}) \cdot \xi_j + \frac 12 <D^2 \phi({\bf 0}) \xi_j, \xi_j> \\
&& \qquad  =\alpha \,h \,  \left |\nabla \phi ({\bf 0})\right | \, \cos \theta_j+\alpha^2\, h^2 \, \partial_{12} \,\phi({\bf 0}) \cos \theta_j \, \sin \theta_j \\
&& \qquad + \frac 12 \,\alpha^2\, h^2 \, \Delta^G_\infty \phi ({\bf 0}) \,\cos^2 \theta_j +  \frac 12 \, \alpha^2\, h^2 \, \Delta^G_1 \phi ({\bf 0})\, \sin \theta_j^2\\
&& \qquad  =\alpha \,  h \, \left |\nabla \phi ({\bf 0})\right | \, \cos \theta_j+\alpha^2\, h^2 \, \partial_{12} \,\phi({\bf 0}) \cos \theta_j \, \sin \theta_j \\
&& \qquad + \frac 12\alpha^2\, h^2 \, \left(\Delta^G_\infty \phi ({\bf 0}) - \Delta^G_1 \phi ({\bf 0})\right) \,\cos^2 \theta_j +  \frac 12\alpha^2\, h^2 \, \Delta^G_1 \phi ({\bf 0}) \\
&& \qquad  =\alpha h \, \left |\nabla \phi ({\bf 0})\right | \, \cos \theta_j+\alpha^2\, h^2 \, \partial_{12} \,\phi({\bf 0}) \cos \theta_j \, \sin \theta_j \\
&& \qquad  \, + \frac 12\alpha^2\, h^2 \left(\Delta^G_\infty \phi ({\bf 0}) - \Delta^G_1 \phi ({\bf 0})\right) \left(\cos^2 \theta_j-\frac 1q\right) +  \frac 12\alpha^2\, h^2 \, \Delta^G_p \phi ({\bf 0}).
\end{eqnarray*}

In this way, due to Lemma~\ref{trivial}, we will prove our conclusion if we show that when $h$ and $\Delta \theta$ tend to zero, the $p$-average of
\begin{equation*}
 \alpha \, h \, \cos \theta_j+ \alpha^2\, h^2 \, \frac{\partial_{12} \,\phi({\bf 0}) }{\left |\nabla \phi ({\bf 0})\right | }\,\cos \theta_j \, \sin \theta_j + \alpha^2 \, h^2  \, \frac{\Delta^G_\infty \phi ({\bf 0}) - \Delta^G_1 \phi ({\bf 0})}{2\,\left |\nabla \phi ({\bf 0})\right | } \,\cos^2 \theta_j\ ,
\end{equation*}
times  $\displaystyle{\frac 2{ \alpha^2 \,h^2}}$\ ,  tends to
$\displaystyle{\frac{\Delta^G_\infty \phi ({\bf 0}) - \Delta^G_1 \phi ({\bf 0})}{\left |\nabla \phi ({\bf 0})\right | }\, \frac 1q}$.

%\begin{equation*}
%\frac{\Delta^G_\infty \phi ({\bf 0}) - \Delta^G_1 \phi ({\bf 0})}{\left |\nabla \phi ({\bf 0})\right | }\, \frac 1q\ .
%\end{equation*}
By definition of $p$-average, we then need to compute the $\hbox{argmin}$ of the function $Z(t)$:
\begin{eqnarray*}
&&Z(t)=\sum_{j=1..M} \left |\alpha\,h \, \cos \theta_j+ \alpha^2\, h^2 \, \frac{\partial_{12} \,\phi({\bf 0})}{\left |\nabla \phi ({\bf 0})\right | }\, \cos \theta_j \, \sin \theta_j \right . \\
&& \qquad \left . +\ \alpha^2\, h^2  \, \frac{\Delta^G_\infty \phi ({\bf 0}) - \Delta^G_1 \phi ({\bf 0})}{2\,\left |\nabla \phi ({\bf 0})\right | } \,\cos^2 \theta_j  -t \right|^p,
\end{eqnarray*}
but:  $\hbox{argmin}~Z(t)=\alpha^2\, h^2\hbox{argmin}~z(t)$, with 
\begin{eqnarray}
&&z(t)=\Delta \theta \sum_{j=1..M}  \left|\cos \theta_j+\alpha  h \, \frac{\partial_{12} \,\phi({\bf 0})}{\left |\nabla \phi ({\bf 0})\right | }\, \cos \theta_j \, \sin \theta_j  \right.\nonumber \\
&& \qquad  \left. +\ \alpha \,  h \, \frac{\Delta^G_\infty \phi ({\bf 0}) - \Delta^G_1 \phi ({\bf 0})}{2\,\left |\nabla \phi ({\bf 0})\right | } \,\cos^2 \theta_j  -\alpha h\,t  \right|^p.
\end{eqnarray}

We set $a=  \partial_{12} \,\phi({\bf 0}) /{\left |\nabla \phi ({\bf 0})\right | }$ and $b=(\Delta^G_\infty \phi ({\bf 0}) - \Delta^G_1 \phi ({\bf 0}))/(2\,\left |\nabla \phi ({\bf 0})\right |)  $, and  recall equation $(\ref{f-der})$ from the Appendix to derive:
\begin{eqnarray*}
z'(t)= -p \,  \Delta \theta  \, \alpha \, h\,  \sum_{j=1}^{M} [(\cos \theta_j+ \alpha \, h \, a \cos \theta_j \, \sin \theta_j  +\alpha\,  h \ b \,\cos^2 \theta_j  -\alpha \, h\,t )\\  
\cdot | \cos \theta_j+ \alpha \, h \, a \cos \theta_j \, \sin \theta_j +\alpha \, h \, b \,\cos^2 \theta_j  - \alpha \, h\,t  |^{p-2}]\ ;
\end{eqnarray*}
we next use $(\ref{f-der2})$  (which holds in the classical sense if $p\geq 2$ and in the weak sense if $1<p<2$) and  the fundamental theorem of calculus to see that for any $d$ and $e$:
$$
-p \, |d+e|^{p-2} \,  (d+e) = - p \, |d|^{p-2} \, d - p (p-1) \int_d^{d+e} |s|^{p-2} \, ds,
$$
hence
\begin{eqnarray*}
&&z'(t)=  \alpha \, h \, \Delta \theta \sum_{j=1..M}  [-p \,| \cos \theta_j |^{p-2} \,  \cos \theta_j \\
&& - p\, (p-1)  \int_{ \displaystyle{\cos \theta_j} }^{ \displaystyle{\cos \theta_j+ \alpha h \, a \cos \theta_j \, \sin \theta_j  +\alpha h \, b\,\cos^2 \theta_j  -\alpha h\,t }}  |s|^{p-2} ds ].
\end{eqnarray*}
And the change of variable $s=\cos \theta_j + \alpha h \, u$ gives
\begin{eqnarray}\label{zder}
&&z'(t)=  \alpha \, h \, \Delta \theta \sum_{j=1..M}  [-p\,    | \cos \theta_j |^{p-2}  \,  \cos \theta_j \\ 
&& - \alpha \,  h \, p\, (p-1) \int_{ \displaystyle{0}}^{ \displaystyle{a \cos \theta_j \, \sin \theta_j  +b\,\cos^2 \theta_j  -\,t }}  |\cos \theta_j + \alpha h \, u|^{p-2} du ]. \nonumber
\end{eqnarray}

We remarked previously that for our original choice of $m$,  if a direction ${\bf r}_j$ is in the set of admissible directions so is its opposite, its orthogonal, as well as its reflections in each of the axes. After rotating the coordinate system, so that $\nabla \phi ({\bf x}) = |\nabla \phi({\bf 0})| \, {\bf e}_1$, we can not make this claim any longer, as for example we lose the reflections with respect to the axes. Nevertheless, given $\theta_j$, we still have $\theta_j+\pi$, so that the sum of the first term in (\ref{zder}) is zero, and we obtain
\begin{equation}\label{zder-01}
z'(t)=  - \alpha \,  h \, p\, (p-1) \int_{ \displaystyle{0}}^{ \displaystyle{a \cos \theta_j \, \sin \theta_j  +b\,\cos^2 \theta_j  -\,t }} \hspace{-1cm} |\cos \theta_j + \alpha h \, u|^{p-2} du . 
\end{equation}

The $\hbox{argmin}$ of $z(t)$, call it $t_0$, is bounded by a constant independent of $\alpha \, h$.
This can be seen by noticing that  $\alpha \, h \, t_0$ is the $p$-average of the values $\{ \cos \theta_j+ \alpha \, h \, a \cos \theta_j \, \sin \theta_j +\alpha \, h \, b \,\cos^2 \theta_j \}$;  but for our choice of angles the $p$-average of the values $\{ \cos \theta_j\}$ is zero (values are symmetric about zero), hence Lemma~\ref{p-stab} implies $|\alpha\, h \, t_0|< \alpha \, h |a| + \alpha \, h |b|$, that is $t_0<|a|+|b|$.

We would like to show that the  $t_0$ is equal to $\displaystyle{\frac bq}$ up to an order of $O(\Delta\theta)+O(\epsilon)$.
To prove our claim, we set  $C(a,b,q)=2 \, \max\left\{|a|+|b|,\displaystyle{\frac {|b|}q}\right\}$, so that both $t_0$ and $\displaystyle{\frac bq}$ belong to the interval $|t|<C(a,b,q)$, and the upper limit of integration in $(\ref{zder-01})$ verifies $|a \cos \theta_j \, \sin \theta_j  +b\,\cos^2 \theta_j  -\,t |<|a|+|b|+C(a,b,q)$ for $|t|<C(a,b,q)$.

On the other hand, when $p\geq 2$, by uniform continuity if $|u|<|a|+|b|+C(a,b,q)$, for any $\epsilon>0$ there is a $\delta_{\epsilon}:= \delta_\epsilon(a,b,q)$ such that for $0<\alpha h<\delta_{\epsilon}$ it holds $|\cos \theta_j +\alpha \,  h u|^{p-2}= |\cos  \theta_j |^{p-2} + O(\epsilon)$.

Therefore, as long as $|t|<C(a,b,q)$,  for  a fixed $\epsilon>0$, there is a $\delta_{\epsilon}$ such that if $0<\alpha h<\delta_{\epsilon}$, we have
\begin{eqnarray}\label{zder-1}
&&z'(t) = -  \alpha^2 h^2\,\Delta \theta\  p\, (p-1)\ \sum_{j=1..M} \int_{ \displaystyle{0}}^{ \displaystyle{a \cos \theta_j \, \sin \theta_j  +b\,\cos^2 \theta_j  -\,t }}  \hspace{-3cm} \left[|\cos \theta_j|^{p-2} + O(\epsilon)\right] ds  \nonumber \\ 
&&\qquad =  -\alpha^2 h^2 \,  \Delta \theta\  p\, (p-1) \sum_{j=1..M}  \, (a \cos \theta_j \, \sin \theta_j)  |\cos \theta_j|^{p-2}  \nonumber \\
&&\quad\qquad -\alpha^2 h^2 \,  \Delta \theta\  p\, (p-1) \sum_{j=1..M}  \, ( b\,\cos^2 \theta_j  -\,t )  |\cos \theta_j|^{p-2} + \alpha^2 \, h^2 O(\epsilon)\ ;\nonumber 
\end{eqnarray}
here we used the fact that $ \Delta \theta \sum_{j=1..M}  = 2 \pi$.

Given an angle $\theta_j$, as mentioned before, we can not assume we still have the angle $\theta_j + \frac \pi2$ as well; nevertheless, we have among our angles an approximation of it up to order $\Delta \theta$. In other words, given a certain direction, in our pool of directions we also have its reflection up to an error of order $O(\Delta \theta)$,  so that 
\begin{eqnarray}\label{zder-2}
&&z'(t)=  \alpha^2 h^2 \, O(\Delta \theta)+  \alpha^2 h^2\, O(\epsilon)\nonumber \\
&& \  \qquad -\ \alpha^2 h^2 \,  \Delta \theta\ p\, (p-1) \sum_{j=1..M}  \, ( b\,\cos^2 \theta_j  -\,t )  |\cos \theta_j|^{p-2}. 
\end{eqnarray}

To proceed in our proof, we recall the elementary equality
\begin{equation*}
\frac{\,\displaystyle{\int_{-\frac \pi 2}^{\frac \pi 2}}(\cos\theta)^p \, d\theta \,}{\displaystyle{\int_{-\frac \pi 2}^{\frac \pi 2}}(\cos\theta)^{p-2}  \, d\theta }=\frac 1q \quad \mbox{ for any }  1<p<\infty,  \mbox{ and }  \frac 1p + \frac 1q =1,
\end{equation*}
which implies 
\begin{eqnarray}\label{zder-3}
&&z'\left(\frac bq\right)=   \alpha^2 h^2 \, O(\Delta \theta)+  \alpha^2 h^2 \, O(\epsilon),  
\end{eqnarray}
since from it we deduce
$$
 \Delta \theta \sum_{j=1..M} \left(b\,\cos^2 \theta_j  -{\frac b q}\right)  |\cos \theta_j|^{p-2}  = O(\Delta \theta).
$$

Next we notice that  for any $c_0>0$ with $| b/q \pm c_0 (O(\Delta \theta)+  O(\epsilon))|< C(a,b,q)$,  we can use $(\ref{zder-2})$ and $(\ref{zder-3})$ to obtain:
\begin{eqnarray}\label{zder-4}
&&\hspace{-0.5cm}z'\left(\displaystyle{\frac bq }\pm  c_0 (O(\Delta \theta)+  O(\epsilon))\right) \nonumber \\
&&\hspace{-0.5cm} =z'\left(\displaystyle{\frac bq}\right) \pm c_0 (O(\Delta \theta)+  O(\epsilon)) \,  \alpha^2 h^2 \,  \Delta \theta\ p\, (p-1) \sum_{j=1..M}   |\cos \theta_j|^{p-2} \nonumber \\
&&\hspace{-0.5cm} = \alpha^2 h^2\, (O(\Delta \theta)+  O(\epsilon)) \, \left (1\pm c_0  \,  \Delta \theta\, p\, (p-1) \sum_{j=1..M}  |\cos \theta_j|^{p-2}\right). 
\end{eqnarray}

We can also  find  a $\Delta_0$, such that for any $\Delta\theta< \Delta_0$, there is a positive constant $c_1$  independent of $\Delta \theta$ and $\epsilon$, for which
 \begin{equation*}
c_1 <\Delta\theta\ p\,(p-1) \sum_{j=1..M} |\cos \theta_j|^{p-2}.
 \end{equation*}
 
Hence,  since for any $\Delta\theta$ and $\epsilon$ small enough to have $|O(\Delta\theta) + O(\epsilon)|<c_1 \left[C(a,b,q)-\displaystyle{\frac{|b|}q}\right]$, we can pick a $c_0$  for which $\displaystyle{\frac 1{c_1}}<c_0<\displaystyle{\frac2 {c_1}}$ and $| b/q \pm c_0 (O(\Delta \theta)+  O(\epsilon))|< C(a,b,q)$,  from $(\ref{zder-4})$ we obtain
\begin{equation*}\label{zder-5}
z'\left(\displaystyle{\frac bq }-  c_0 \left|O(\Delta \theta)+  O(\epsilon)\right|\right) <0\ ,
 \quad z'\left(\displaystyle{\frac bq }+  c_0 \left|O(\Delta \theta)+  O(\epsilon)\right|\right) >0\ .
\end{equation*}

Recalling that  $t_0=\hbox{argmin } z(t)$ is its only critical point, and $z'(t)>0$ for every $t>t_0$, while $z'(t)<0$  for $t<t_0$, see Remark~{\ref{r7.1}}, we conclude that 
\begin{equation}
\left|\mbox{argmin } z(t) - \frac bq \right| \leq   \frac 2{c_1} \left|O(\Delta \theta)+  O(\epsilon) \right|,
\end{equation}
for every $\Delta\theta$ and $\epsilon$ small enough.

Therefore, we have that for $\epsilon$ and $\Delta\theta$ small enough, there is a $\delta_{\epsilon}$ such that if $0<\alpha h<\delta_{\epsilon}$ then
$$
\left |\frac 2{\alpha^2 \, h^2} \mbox{ argmin } Z(t) -\frac{\Delta^G_\infty \phi ({\bf 0}) - \Delta^G_1 \phi ({\bf 0})}{\left |\nabla \phi ({\bf 0})\right | }\, \frac 1q \right | \leq |O(\Delta \theta)+   O(\epsilon)|, 
$$
and the theorem follows for the case $\nabla \phi({\bf x})  \neq 0. $

Assume $\nabla \phi({\bf x})  = 0 $, then by Definition~\ref{G-visco-sol} we need to show that the
$p$-average of  $\displaystyle{\frac 2{\alpha^2\,h^2}\left( \frac 12 <D^2 \phi({\bf 0}) \xi_j, \xi_j>\right)}$ tends to $\Delta^G_2 \phi({\bf 0})$ as $h$ and $\Delta \theta$ tend to 0.  By our choice of angles,  $\xi_j = \alpha \, h \, (\cos j\Delta \theta, \sin j\Delta \theta)$, so that
\begin{eqnarray}
&& \frac 2{\alpha^2\,h^2}\left( \frac 12 <D^2 \phi({\bf 0}) \xi_j, \xi_j>\right) = \partial_{11}  \phi({\bf 0}) \cos^2( j\Delta\theta) \nonumber \\
&&  \qquad + \partial_{22}  \phi({\bf 0}) \sin^2( j\Delta\theta) + 2 \, \partial_{12}  \phi({\bf 0}) \cos( j\Delta\theta)\, \sin( j\Delta\theta) \nonumber \\
&& \qquad = \frac 12 \Delta_2  \phi({\bf 0}) + \frac 12(\partial_{11}  \phi({\bf 0}) - \partial_{22} \phi({\bf 0})) \cos ( 2 j\Delta\theta)   \nonumber \\
&& \qquad + 2 \, \partial_{12}  \phi({\bf 0}) \cos( j\Delta\theta)\, \sin( j\Delta\theta) \nonumber, 
\end{eqnarray}
but  by our assumptions if $\theta $ is in our selection of angles, so is $\theta + \frac {\pi}2$, thus the $p$-average of 
$$
\frac 12 (\partial_{11}  \phi({\bf 0}) - \partial_{22} \phi({\bf 0})) \cos ( 2 j\Delta\theta)  + 2 \, \partial_{12}  \phi({\bf 0}) \cos( j\Delta\theta)\,  \sin( j\Delta\theta) 
$$
is zero, being a set of symmetric data with respect to 0. Therefore, using (\ref{transla}),
$$
A_p\left[  \frac 2{\alpha^2\,h^2}\left( \frac 12 <D^2 \phi({\bf 0}) \xi_j, \xi_j>\right) \right]= \frac 12 \Delta_2  \phi({\bf 0}) =  \Delta_2^G  \phi({\bf 0}) ,
$$
and the theorem is proven.
\end{proof}

We are now ready to show consistency of our approximation scheme. Before doing so we remark that although the definition of consistency we use is slightly different from the one given in \cite{BS91}, their convergent result applies also for this formulation. 

\begin{theo}\label{consistency}
Let  $p\geq2$. Our approximation scheme is consistent, that is  for all   ${\bf x}\in\overline  \Omega$ and $\phi\in C^\infty(\overline \Omega)$, we have that 
\begin{equation}
 \limsup_{\underset{\underset{\xi \to 0}{{\bf y} \to {\bf x}}}{\rho\to 0}} S(\rho, {\bf y}, \phi({\bf y})+\xi, \phi+\xi) \leq \limsup_{\underset{{\bf y}\in \overline \Omega}{{\bf y} \to {\bf x}}} G_p(D^2\phi({\bf y}),D\phi({\bf y}),\phi({\bf y}),{\bf y})
 \end{equation}
(where, as before,  $\rho=(h,\Delta\theta)$) and
 \begin{equation}
\liminf_{\underset{\underset{\xi \to 0}{{\bf y} \to {\bf x}}}{\rho\to 0}} S(\rho, {\bf y}, \phi({\bf y})+\xi, \phi+\xi) \geq \liminf_{\underset{{\bf y}\in \overline \Omega}{{\bf y} \to {\bf x}}} G_p(D^2\phi({\bf y}),D\phi({\bf y}),\phi({\bf y}),{\bf y}).
\end{equation}

\end{theo}
\begin{proof}
If ${\bf x} \in \Omega$ both statements are a consequence of our previous theorem, i.e. Theorem~\ref{consist}.
On the other hand, if ${\bf x} \in \partial \Omega$, we have that 
\begin{eqnarray*}
&& \limsup_{\underset{{\bf y}\in \overline \Omega}{{\bf y} \to {\bf x}}} G_p(D^2\phi({\bf y}),D\phi({\bf y}),\phi({\bf y}),{\bf y}) \\
&& \qquad = \max\left(H_p(D^2\phi({\bf x}),D\phi({\bf x}),{\bf x}), \phi({\bf x})-F({\bf x})\right),
 \end{eqnarray*}
 while  
 \begin{eqnarray*}
&& \liminf_{\underset{{\bf y}\in \overline \Omega}{{\bf y} \to {\bf x}}} G_p(D^2\phi({\bf y}),D\phi({\bf y}),\phi({\bf y}),{\bf y}) \\
&& \qquad = \min\left(H_p(D^2\phi({\bf x}),D\phi({\bf x}),{\bf x}), \phi({\bf x})-F({\bf x})\right),
 \end{eqnarray*}
and the theorem follows by the definition of $S$.
\end{proof}

However, as far as an explicit scheme in time is applied  to 
\begin{equation}
u_t+G_p(D^2u({\bf y}),Du({\bf y}),u({\bf y}),{\bf y})=0,
\end{equation}
the consistency of the scheme $S$ with respect to  the stationary nonlinear operator $G_p$  implies the consistency for the evolutive operator as in \cite{BS91}.\\
In fact, by applying the Euler approximation in time we have, for a given initial condition $u^0$, the explicit time marching scheme
\begin{equation}\label{evolGp}
u^{n+1}=u^n-\Delta t \ S(\rho, {\bf y}, u^n({\bf y}), u^n),
\end{equation}
which implies, taking $\Delta t=|\rho|$,
\[
 \frac{u^{n+1}-u^{n}}{|\rho|} = -S(\rho, {\bf y}, u^n({\bf y}), u^n).
\]
Passing to the limit for $|\rho|$ which tends to 0, we get consistency in the usual sense.

\begin{theo}\label{stability}
Let $f=0$. For all $h>0, \Delta \theta >0$, there exists a solution $u_\rho \in L^\infty(\overline\Omega)$ of (\ref{ellipticscheme}) such that $\displaystyle{||u_\rho||_{L^\infty(\overline\Omega)}\leq ||F||_{L^\infty(\partial\Omega)}}$.
\end{theo}
\begin{proof}
We consider the operator $E_\rho:L^\infty(\overline \Omega))\to L^\infty(\overline \Omega))$ defined as
\begin{equation*}
E_\rho(u)({\bf x})= \left \{ \begin{array}{ll}
 A_p(C_h^{\Delta \theta}( {\bf x},u;\alpha))  \, &\mbox{ if }  {\bf x} \in \Omega\\ \\
F({\bf x})  \,&\mbox{ if } {\bf x} \in \partial \Omega, 
\end{array} \right. \nonumber
\end{equation*}
and notice that thanks to Lemma~\ref{trivial} in the Appendix, we have
\begin{equation}\label{soardi}
||E_\rho u||_{L^\infty(\overline\Omega)}\leq \max(||u||_{L^\infty(\overline\Omega)},||F||_{L^\infty(\partial\Omega)}).
\end{equation}
Additionally,  $E_\rho$ is a nonexpansive operator in the $L^\infty$ norm, that is
\begin{equation}\label{nonexp}
||E_\rho(u) - E_\rho(v)||_{L^\infty(\overline\Omega)}\leq || u-v||_{L^\infty(\overline\Omega)},
\end{equation}
since if ${\bf x} \in \Omega$, by Lemma~\ref{p-stab} we know that
\begin{equation*}
|E_\rho(u)({\bf x}) - E_\rho(v)({\bf x})|= 
 |A_p(C_h^{\Delta \theta}( {\bf x},u;\alpha))-A_p(C_h^{\Delta \theta}( {\bf x},v;\alpha))| \leq || u-v||_{L^\infty(\overline\Omega)},  
\end{equation*}
while
$$
|E_\rho(u)({\bf x}) - E_\rho(v)({\bf x})|= 0, \, \mbox{ if } {\bf x} \in \partial \Omega.
$$

If we consider $B(0,||F||_{L^\infty(\partial\Omega)})$ to be the sphere centered at  the zero function  and of radius $||F||_{L^\infty(\partial\Omega)}$, equations (\ref{soardi}) and (\ref{nonexp}) imply that $E_\rho$ is a nonexpansive operator mapping the closed sphere $\overline{B(0,||F||_{L^\infty(\partial\Omega)})} \subset L^\infty(\overline\Omega)$ into itself. Therefore, by a classical fixed point theorem result (see Corollary~1 and Remark thereafter in \cite{So79}) we conclude that $E_\rho$ has a fixed point in $\overline{B(0,||F||_{L^\infty(\partial\Omega)})}$, and the theorem follows.
\end{proof}

\begin{theo}\label{conve}
Assume $f\equiv 0$ and $p\geq 2$, then the solution $u_\rho$ of the appro\-ximation scheme~(\ref{ellipticscheme}) converges  as $\rho=(h,\Delta \theta) \to (0,0)$ to the viscosity solution $u$ of (\ref{Fp-game}).
\end{theo}

\begin{proof} \, For $f=0$, we know that (\ref{Fp-game})  has a unique bounded viscosity solution $u$, with $\displaystyle{||u||_{L^\infty(\overline\Omega)}\leq ||F||_{L^\infty(\partial\Omega)}}$ and that a comparison principle holds (see \cite{M,JLM} for details).   Therefore, thanks to Theorems~\ref{monotonicity}, \ref{consistency} and \ref{stability}, we can apply Theorem~2.1 in \cite{BS91}.
\end{proof}

%%%%%%%%%%%%%%%%%%%%%%%%%%%%%%%%%%%%%%%%%%%%%%%%%%%%%%%%%%%%%%%%%
%%%%%%%%%%%%%%   NUMERICAL IMPLEMENTATION
%%%%%%%%%%%%%%%%%%%%%%%%%%%%%%%%%%%%%%%%%%%%%%%%%%%%%%%%%%%%%%%%%
\section{Numerical implementation}\label{parabolic}

Let us now introduce in $\Omega$ a structured grid according to a space discretization parameter $h$, with $0<h<1$, and denote by $\{{\bf x}_j\}_{j=1..N}$ its nodes. 
An important step to  implement our approximation scheme for a fixed angular resolution $\Delta \theta$ is the way we reconstruct the values $C_h^{\Delta \theta}( {\bf x}, u;\alpha )$ in   \eqref{approxGp}, starting from the known values of $u$ at the nodes of the grid. This is done via interpolation, but in order to obtain a convergence result we must restrict to monotone interpolation techniques. This means that denoted by $u$ a function defined in a domain $D$ and by $I[u](\cdot)$ its local interpolation based on the values at the nodes,  we can only use an interpolation operator such that
\begin{equation}
m=\min_{x\in D} u(x) \le I[u](x)\le \max_{x\in D} u(x)=M.
\end{equation}
Note that linear and bilinear interpolation in $\R^2$ are monotone interpolation. Another important property of bilinear interpolation is that it is translation invariant, i.e. given a constant $\delta$ we have 
\begin{equation}
I[u+\delta](x) = I[u]+\delta.
\end{equation}
This property will guarantee that the resulting scheme is also invariant with respect to the addition of constants.

We implement our approximation scheme on a structured grid using the classical time marching approximation \eqref{evolGp}. Since our solution $u^n$ is characterized by its values at the nodes, we prefer to work in the space ${\R}^N$, denoting by ${\bf u}^n$ the iterate at step $n$, that is the vector of components $u^n_j=u^n({\bf x}_j)$; given an initial condition ${\bf u}^0\in{\R}^N$ we then consider the iterative scheme
\begin{equation}\label{scheme}
\hspace{-0.17cm} u^{n+1}_j \hspace{-0.1cm}= \hspace{-0.1cm} \left \{ \begin{array}{ll}
\hspace{-0.20cm} u^{n}_j +\displaystyle{ \frac{2\, \Delta t}{ \alpha_j^2 \, h^2}} \left[ A_p({\widehat C}_h^{\Delta \theta}( {\bf x}_j,{\bf u}^{n};\alpha_j )) - \, u^{n}_j\right]+\Delta t \ f({\bf x}_j)  &\mbox{} {\bf x}_j\in \Omega,  \\ \\
\hspace{-0.20cm} F({\bf x}_j) &\mbox{}{\bf x}_j\in \partial\Omega;
\end{array} \right.
\end{equation}
%\begin{eqnarray}\label{scheme}
%&& u^{n+1}({\bf x}_j)=u^{n}({\bf x}_j)+\Delta t \, f({\bf x}_j) \nonumber  \\
%&&\, \qquad+\displaystyle{ \frac{2\, \Delta t}{ \alpha_j^2 \, h^2}} \left[ A_p({\widehat C}_h^{\Delta \theta}( {\bf x}_j,u^{n};\alpha_j )) - \, u^{n}({\bf x}_j)\right] \, \mbox{ if } {\bf x}_j\in \Omega,  \\
%&& u^{n+1}({\bf x}_j)=F({\bf x}_j) \qquad  \mbox{ if } {\bf x}_j\in \partial\Omega; \nonumber
%\end{eqnarray}
if for a given integer $m$ we assume $\Delta\theta=\pi/(2\,m)$, now
\begin{equation}\label{stencil}
{\widehat C}_h^{\Delta \theta}( {\bf x}_j, {\bf u}^n;\alpha_j )=\{w({\bf x}_j^i,u^n), \, i=0,...,4m-1\},
\end{equation}
 where ${\bf x}_j^i={\bf x}_j + h \,\alpha_j\, {\bf r}_i$ (with ${\bf r}_i = \left(\cos i\Delta\theta, \sin i \Delta\theta\right)$), and 
$w({\bf x}_j^i,u^n)$ is the value  of the function $u^n$ at ${\bf x}_j^i$ obtained by using bilinear interpolation on the four closest grid-points (in fact this is the main difference with respect to $C_h^{\Delta \theta}$). Here, $\alpha_j$ is a parameter that may vary at each grid point ${\bf x}_j$, and which verifies $0<\alpha_*<\alpha_j <dist({\bf x}_j, \partial \Omega)<d_\Omega$, for some constant $\alpha_*$ and  for $d_\Omega$, the diameter of $\Omega$.  In other words, $\alpha_j$ is a function of $j$, which is uniformly bounded independently of $h$ from above and below, and  such that ${\bf x}_j^i$ belongs to the computational domain whenever $0<h<1$.

More precisely, if ${\bf x}=(x,y)$ is a point contained in a grid cell with vertices $(x_k,y_l)$ (the lower left corner), $(x_{k+1},y_l), (x_k,y_{l+1})$ and $(x_{k+1},y_{l+1})$, and we denote by $u_{k,l}=u(x_k,y_l)$, etc., the known values of a function $u$ on them, the bilinear interpolation computes in $(x,y)$ the second order polynomial
\begin{equation*}   
I[u](x,y)=axy+bx+cy+d  ,
\end{equation*}
where the coefficients $a$, $b$, $c$ and $d$ can be determined
solving the linear system  $4\times 4$ which corresponds to the four height
conditions at the four vertices of the cell.  Those values can also be written as linear combinations of
the values of $u$ at the vertices of the cell, i.e.
\begin{equation*} 
I[u](x,y)=\lambda_{k,l}u_{k,l}+\lambda_{k+1,l}u_{k+1,l}+\lambda_{k,l+1}u_{k,l+1} 
         +\lambda_{k+1,l+1}u_{k+1,l+1},  
\end{equation*}
where the coefficients  are given by 
\begin{equation*}  
\lambda_{k,l}=(x_{k+1}-x)(y_{l+1}-y)/V , 
\end{equation*}
\begin{equation*}  
\lambda_{k+1,l}=(x-x_{k})(y_{l+1}-y)/V  ,
\end{equation*}
\begin{equation*}  
\lambda_{k,l+1}=(x_{k+1}-x)(y-y_l)/V  ,
\end{equation*}
\begin{equation*}  
\lambda_{k+1,l+1}=(x-x_k)(y-y_l)/V  ,
\end{equation*}
and $V:=(x_{k+1}-x_k)(y_{l+1}-y_l)$ is the area of the cell, i.e. $V=h^2$ for our uniform grid (see for example \cite{FT09}). 

The $p$-average in (\ref{scheme}) is then computed through the $4m$ values of $I[u^n]$ on a set of equally distributed points on the sphere of center ${\bf x}_j$ and radius $h\alpha_j$.

%{\color{red}
%{\bf WARNING : we can also add some pictures if necessary}
%\begin{figure}
%\begin{center}
%\mbox{\epsfig{file=maglia.eps}}
%\caption{ The circle spanned by $x_{ij}-ka$ and the grid. }
%\label{fretic}
%\end{center}
%\end{figure}

%\begin{figure}
%\begin{center}
%\mbox{\epsfig{file=retic.eps}}
%\caption{ The circle spanned by $x_{ij}-ka$ and the grid. }
%\label{fretic}
%\end{center}
%\end{figure}

%Let $G$ be a grid cell (see Figure \ref{fretic}) with vertices at
%In practice  (E' ancora vero per il multicircle ?), it is useful to write the $\lambda$ coefficients as functions of the $\theta$ angle: 
%\begin{equation}        
%\begin{array}{ll}
%  &\lambda_{i,j}=(k-k\cos\theta)(k-k\sin\theta)/k^2=(1-\cos\theta)(1-\sin\theta)\\ \\
%  &\lambda_{i+1,j}=(k\cos\theta)(k-k\sin\theta)/k^2=\cos\theta(1-\sin\theta) \\\\  
%  &\lambda_{i,j+1}=(k-k\cos\theta)(k\sin\theta)/k^2=\sin\theta(1-\cos\theta) \\\\  
%  &\lambda_{i+1,j+1}=(k\cos\theta)(k\sin\theta)/k^2=\sin\theta\cos\theta.\\ \\
%          \end{array}
%\end{equation}

%
%}

Then the  approximate solution of the Dirichlet problem $(\ref{p-game})$ for the game $p$-Laplacian is  computed  by using the scheme above and running it  until the stopping rule 
\begin{equation}\label{stop}
E_n = \max_j |u_j^{n+1}-u_j^n|\le \epsilon
\end{equation}
is satisfied, for a given tolerance $\epsilon$. Another option is to use the simpler iterative scheme obtained by setting $\displaystyle{ \frac{2\, \Delta t}{ \alpha_j^2 \, h^2}=1}$ in $(\ref{scheme})$, that is 
\begin{equation}\label{scheme-ellip}
u^{n+1}_j \hspace{-0.1cm}= \hspace{-0.1cm}\left \{ \begin{array}{ll}
\hspace{-0.2cm} A_p({\widehat C}_h^{\Delta \theta}( {\bf x}_j,{\bf u}^{n};\alpha_j )) +\Delta t \, f({\bf x}_j)  &\mbox{ if } {\bf x}_j\in \Omega,  \\ \\ \\
\hspace{-0.2cm} F({\bf x}_j) &\mbox{ if }{\bf x}_j\in \partial\Omega;
\end{array} \right.
\end{equation}
until convergence (i.e. until \eqref{stop} is satisfied).
\ \

We next show that for a suitable initial configuration the iteration generated by (\ref{scheme}) converges if $f\equiv0$. Although we can not claim  that this proves convergence to the solution of our approximation scheme,  the numerical tests that we present in Section~\ref{numtest} show convergence to the correct viscosity solution in all cases were the exact solution is known, also for $f\neq 0$. Let us rewrite our numerical iteration as 
\begin{equation}\label{Srho}
{\bf u}^{n+1}=T_\rho({\bf u}^n),
\end{equation}
where  for  ${\bf u}\in{\R}^N$, we have set
\begin{eqnarray}\label{vscheme}
(T_\rho({\bf u}))_j=\left \{ \begin{array}{ll}
\hspace{-0.2cm}u_j +\displaystyle{ \frac{2\, \Delta t}{\alpha_j^2\, h^2}} \left[ A_p({\widehat C}_h^{\Delta \theta}( {\bf x}_j, {\bf u};\alpha_j )) - \, u_j\right] +\Delta t \, f({\bf x}_j), &{\bf x}_j \in \Omega,\\ \\ 
\hspace{-0.2cm}F({\bf x}_j ) &{\bf x}_j \in \partial \Omega.
\end{array} \right.
\end{eqnarray}

It is not difficult to show that for a given grid, the fact that the boundary nodes have a fixed value over each iteration prevents the numerical solution from blowing up even in the presence of the source term $f$. On the other hand, for $f \neq 0$ the bound depends on the grid size and $f$ jointly, in a manner for which we don't have an independent bound. Instead, if  $f\equiv 0$ it is very easy to derive that the initial condition ${\bf u}^0$ will provide a bound for any subsequent iteration as the following result shows.

\begin{theo}\label{stab}
Let $\displaystyle{\frac{2\, \Delta t}{ \alpha_*^2\, h^2}\leq 1}$. Assume $f\equiv 0$, then for $n\geq1 $ it holds
$$
\sup_{j=1..N} |u_j^{n}| \leq  \sup_{j=1..N} |u_j^{n-1}| \leq \sup_{j=1..N} |u_j^{0}|,
$$
where ${\bf u}^n$ is defined by (\ref{Srho}).
\end{theo}

\begin{proof} It is enough to look at the internal nodes, since the approximation at the boundary nodes has fixed values. For $j$ fixed, the set  ${\widehat C}_h^{\Delta \theta}( {\bf x}_j, {\bf u}^n;\alpha_j )$ consists of values $w({\bf x}_j^i,u^n)$ computed by using bilinear interpolation  of $u^n$. Hence, each $w({\bf x}_j^i,u^n)$ is controlled by the values $u^n_k$, and  Lemma~\ref{trivial} in the Appendix implies that 
\begin{eqnarray*}
&&|u_j^{n}| =\left|(T_\rho({\bf u}^{n-1}))_j \right| \\
&&\quad =\left|\left(1- \frac{2\, \Delta t}{\alpha_j^2\, h^2}\right) u_j^{n-1} + \frac{2\, \Delta t}{\alpha_j^2\, h^2} A_p({\widehat C}_h^{\Delta \theta}( {\bf x}_j, {\bf u}^{n-1};\alpha_j ))\right|\\
&&\quad  \leq \left(1- \frac{2\, \Delta t}{\alpha_j^2 \, h^2}\right) | u^{n-1}_j|+ \frac{2\, \Delta t}{\alpha_j^2\,h^2}  \sup_{j..N} |u_j^{n-1}|,
\end{eqnarray*}
and stability follows.
\end{proof}

The next theorem shows that if the initial condition is appropriately chosen, then the iterates generated by our schemes are point-wise increasing; this fact together with the previous stability result implies that they pointwise converge if $f\equiv0$.

\begin{theo}\label{increa}
Let $\displaystyle{\frac{2\, \Delta t}{ \alpha_*^2\, h^2}\leq 1}$. For $n\geq1$,  there exists an initial condition, ${\bf u}^0$, for which the iterations generated by the scheme~(\ref{Srho}) verify 
$$
u^n_j\geq u^{n-1}_j  \mbox{ for any }  j=1..N, \mbox{ and } n\geq 1.
$$
\end{theo}

\begin{proof} We choose as initial condition:
\begin{equation}
u^0_j= \left \{ \begin{array}{ll}
\displaystyle{\min_{\partial \Omega}} \, F  \, & \mbox{ if } {\bf x}_j \in \Omega,  \\ \\
F({\bf x}_j) &  \mbox{ if } {\bf x}_j \in \partial \Omega,  
\end{array} \right. \nonumber
\end{equation}
and since $u^0_j\geq\displaystyle{\min_{\partial \Omega}} \, \, F$ for any $j$ we have
$$
A_p({\widehat C}_h^{\Delta \theta}( {\bf x}_j, {\bf u}^0;\alpha_j )) \geq \displaystyle{\min_{\partial \Omega}} \,  F \quad \mbox{ if }  {\bf x}_j \in \Omega.
$$
Therefore, 
\begin{eqnarray*}
&&u^{1}_j = \left(1-\displaystyle{ \frac{2\, \Delta t}{\alpha_j^2\, h^2}} \right) u_j^0+\displaystyle{ \frac{2\, \Delta t}{\alpha_j^2 \, h^2}} A_p({\widehat C}_h^{\Delta \theta}( {\bf x}_j, {\bf u}^0;\alpha_j )) + \Delta t \,  f(x_j) \\
&&\qquad \geq \displaystyle{\min_{\partial \Omega}} \,  F + \Delta t \, f(x_j) \geq u^0_j, \qquad  \mbox{ if }  {\bf x}_j \in \Omega,
\end{eqnarray*}
since $f\geq 0.$ Then, we conclude $u^1_j\geq u^0_j$ for every $j$, as well as
$$
A_p({\widehat C}_h^{\Delta \theta}( {\bf x}_j, {\bf u}^1;\alpha_j )) \geq A_p({\widehat C}_h^{\Delta \theta}( {\bf x}_j, {\bf u}^0;\alpha_j )), \mbox{ if }  {\bf x}_j \in \Omega,
$$
thanks to Lemma~\ref{p-mono}. Inserting this last inequality in the definition of $u^2_j$ leads to
\begin{eqnarray*}
&&u^2_j=  \left(1-\displaystyle{ \frac{2\, \Delta t}{\alpha_j^2\, h^2}} \right) u_j^1 +\displaystyle{ \frac{2\, \Delta t}{\alpha_j ^2\, h^2}} A_p({\widehat C}_h^{\Delta \theta}( {\bf x}_j, {\bf u}^1;\alpha_j ))  +\Delta t \, f(x_j) \\
&&\ \quad \geq  \left(1-\displaystyle{ \frac{2\, \Delta t}{\alpha_j^2\, h^2}} \right) u_j^0 +\displaystyle{ \frac{2\, \Delta t}{\alpha_j^2 \, h^2}} A_p({\widehat C}_h^{\Delta \theta}( {\bf x}_j, {\bf u}^0; \alpha_j))+  \Delta t \, f(x_j)  \\
&&\ \quad = u^1_j, \qquad \mbox{ if }  {\bf x}_j \in \Omega,
\end{eqnarray*}
that is  $u^2_j\geq u^1_j$ for every $j$. We then obtain the desired conclusion by induction.
\end{proof}

\begin{rem}
In the case $f\equiv 0$, if we pick as initial condition
\begin{equation*}
u^0_j= \left \{ \begin{array}{ll}
\displaystyle{\max_{\partial \Omega}} \, F  \, & \mbox{ if } {\bf x}_j \in \Omega,  \\ \\
F({\bf x}_j) &  \mbox{ if } {\bf x}_j \in \partial \Omega,  
\end{array} \right. \nonumber
\end{equation*}
and follow the steps of the proof of Theorem~\ref{increa}, we obtain a non-increasing sequence, which dominates element by element the sequence in Theorem~\ref{increa}. Note that, if these two sequences converge to the same limit $\overline u$, choosing as initial iteration a ${\bf u}^0$ such that $u^0_j$ is between the minimum and maximum values of $F$,  we would have a sequence converging again to $\overline u$.  Since for $f\equiv 0$ we know that problem (\ref{p-game}) has a unique viscosity solution, if we could show that our numerical implementation converges to it then any initial condition such that $\min_{\partial \Omega} \, F \leq  u^0_j \leq \max_{\partial \Omega} \, F$ would produce a sequence converging to the viscosity solution.
Let us also observe that the choice of a monotone interpolation in the numerical implementation guarantees that similar bounds also apply to the interpolation of our initial condition and to all the elements in the sequence.
\end{rem}
%%%%%%%%%%%%%%%%%%%%%%%%%%%%%%%%%%%%%%%%%%%%%%%%%%%%%%%%%%%%%%%%%
%%%%%%%%%%%%%%   NUMERICAL TESTS
%%%%%%%%%%%%%%%%%%%%%%%%%%%%%%%%%%%%%%%%%%%%%%%%%%%%%%%%%%%%%%%%%
\section{Numerical tests}\label{numtest}

We present in this section some experiments obtained with our numerical implementation coded in MATLAB, and executed on a  MacBook Pro desktop machine with a 2.2 GHz Intel Core 2 Duo processor. As described in Section~5, we have used structured uniform grids, and the values of the approximate solution at the points in ${\widehat C}_h^{\Delta\theta}$ of (\ref{stencil}) have been computed via bilinear interpolation using the four closest grid points. If one of the points lies on a line joining two grid points, the bilinear interpolation reduces to the linear interpolation between them. To compute the $p$-average at each node we used the Newton Bracketing method for minimization of convex functions, \cite{LB}, applied to the function $\displaystyle{g(s):=\left( Q(s,S)\right)^{\frac 1p}}$, where $Q(s,S)$ is defined as in equation (\ref{f-average}). We believe that an optimization of this part of the procedure could improve the speed of calculations. 

Depending on the examples we have picked different values for the parameter $\alpha_j$ at different grid points. To better illustrate our choices, we introduce the following simple definition, where $d_j$ denotes the distance of ${\bf x_j}$ from $\partial \Omega$ along the grid lines, and therefore $\frac {d_j}h$ is an integer greater or equal to 1 (see Figure~1).

\begin{defi}
An iteration generated by (\ref{scheme}) is called  {\em n-level circles iteration} if for every $j$ the parameter $\alpha_j$ is chosen so that $\alpha_j=\beta \min(n,\frac {d_j}h)$, for a given $0<\beta\leq 1$. \end{defi}

\begin{figure}
\begin{center}
\resizebox{3in}{!}{\includegraphics{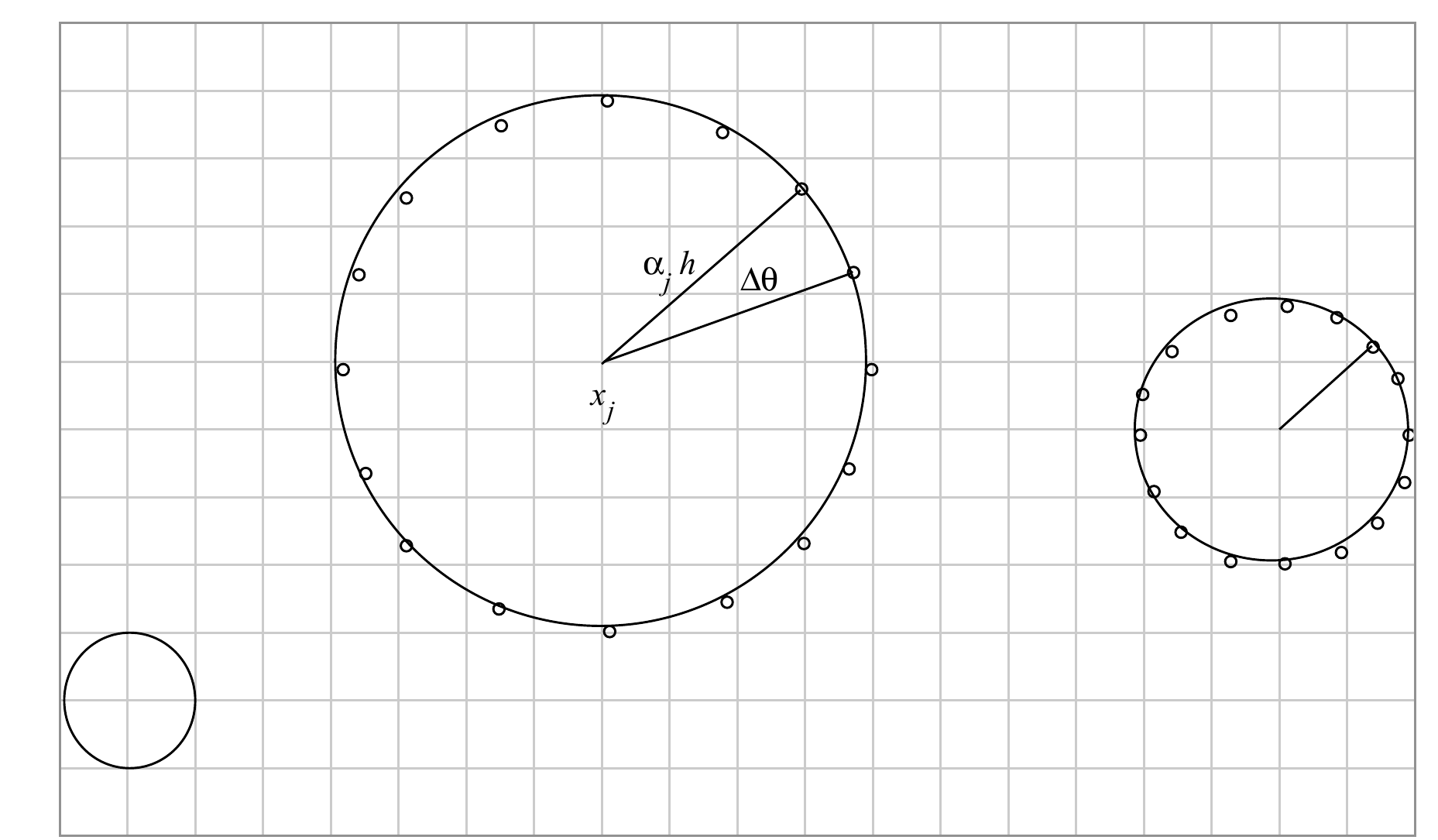}}
\end{center}
\caption{The 4-level circles stencil, with $\beta=1$ }
\end{figure} 

\medskip{\bf Test 1.} For comparison, first of all we run the schemes of the previous section on the examples for the $\infty$-Laplacian presented in \cite{O04}, using the same specifications provided there (in this case convergence is assumed to be reached when $E_n\le 2h10^{-2}$, where $E_n$ is the quantity defined in (\ref{stop})). 

We have applied scheme $(\ref{scheme})$ to the case $\Omega=(-1,1)\times(-1,1)$, $p=~\infty$, $f\equiv~0, F(x,y)= |x|^{4/3} -|y|^{4/3}$,  and the results are summarized in Table~1 below. The explicit solution of this problem is the well-known Aronsson function $u(x,y)=|x|^{4/3} -|y|^{4/3}$. We denote by $N^2$ (so that $h=2/N$) the total number of nodes in the square grid, and by $n$ the number of iterations (reported between parentheses), while the error is given in the maximum norm and a {\it 2-level} circles iteration is used, for $\beta=0.99$. The initial condition ${\bf u}^0$ was assumed to be a perturbation of the exact solution (of order $\pm 20\%$).The total number of grid points and the number of directions used to compute ${\widehat C}_h^{\Delta\theta}$ have been choosen to make our tests comparable  with the ones for the 17-point and the 25-point stencils presented in Table~2 of \cite[pg. 1227]{O04}. Note that the errors decrease on the rows and on the columns in a regular way and that also the number of iterations  decreases if we increase the number of directions and $N$ simultaneously.  Our numerical results show essentially the same accuracy of those of \cite{O04}. We only remark that in our case, in order to significantly reduce the error when $h$ tend to zero, we should also increase the number of directions. 

\begin{table}[ht]\label{tab1}
\small
\begin{center}
\begin{tabular}{|c|c|c|c|c|c|}
\hline
Dir.&$N=41$&$N=81$&$N=161$&$N=241$&$N=401$\\
\hline
4&0.1105 (250)&0.0765 (448)&0.0373 (584)&0.0225 (589)&0.0122 (621)\\
\hline
8&0.0274 (80)&0.0182 (161)&0.0084 (214)&0.0069 (190)&0.0048 (188)\\
\hline
16&0.0084 (54)&0.0070 (75)&0.0043 (105)& 0.0033 (108)&0.0023 (112)\\
\hline
24&0.0088 (57)&0.0081 (73)&0.0050 (91)& 0.0035 (103)&0.0024 (107)\\
\hline
\end{tabular}
\caption{$L^\infty$- errors and iterations (in parentheses) for Test  1 on Aronsson function}
\end{center}
\end{table}

In Figure~2 one can see the numerical solutions and its contour plots obtained, when again $\Omega=(-1,1)\times(-1,1)$, $p=\infty$ and $f\equiv0$, for three different choices of the boundary data. These computations are included for comparison with  Figure~2  in \cite[pg. 1228]{O04}.  In the first two cases, we consider $F(x,y)=|x|^2 \, |y|^2$ and  $F(x,y)=x^3-3xy^2$. We use 24 controls, {\it 2-level} circles iterations, $\beta=0.99$, and  a grid with $401^2$ nodes.  The initial condition for this grid is generated by a multi-resolution type approach. More specifically, we start from a 21 by 21 coarse grid with initial values set to zero in the interior nodes; after a few iterations we interpolate the numerical solution on a finer grid and repeat the procedure up to the desired resolution. Although this start up procedure requires some time, but significantly improves the rate of convergence. The approximations in Figure~2 for these two cases took $n=15$ and $n=47$ iterations, respectively, to converge. In the third example, $F$ is the characteristic function of the point $(1,0)$. We proceed as described above, but we use 24 controls on a  {\it 4-level} circles iteration, again for $\beta=0.99$. In this test, the approximation took $n=2428$ iterations to converge with the required accuracy.

\begin{figure}
\begin{center}
\begin{tabular}{cc}
\resizebox{2.1in}{!}{\includegraphics{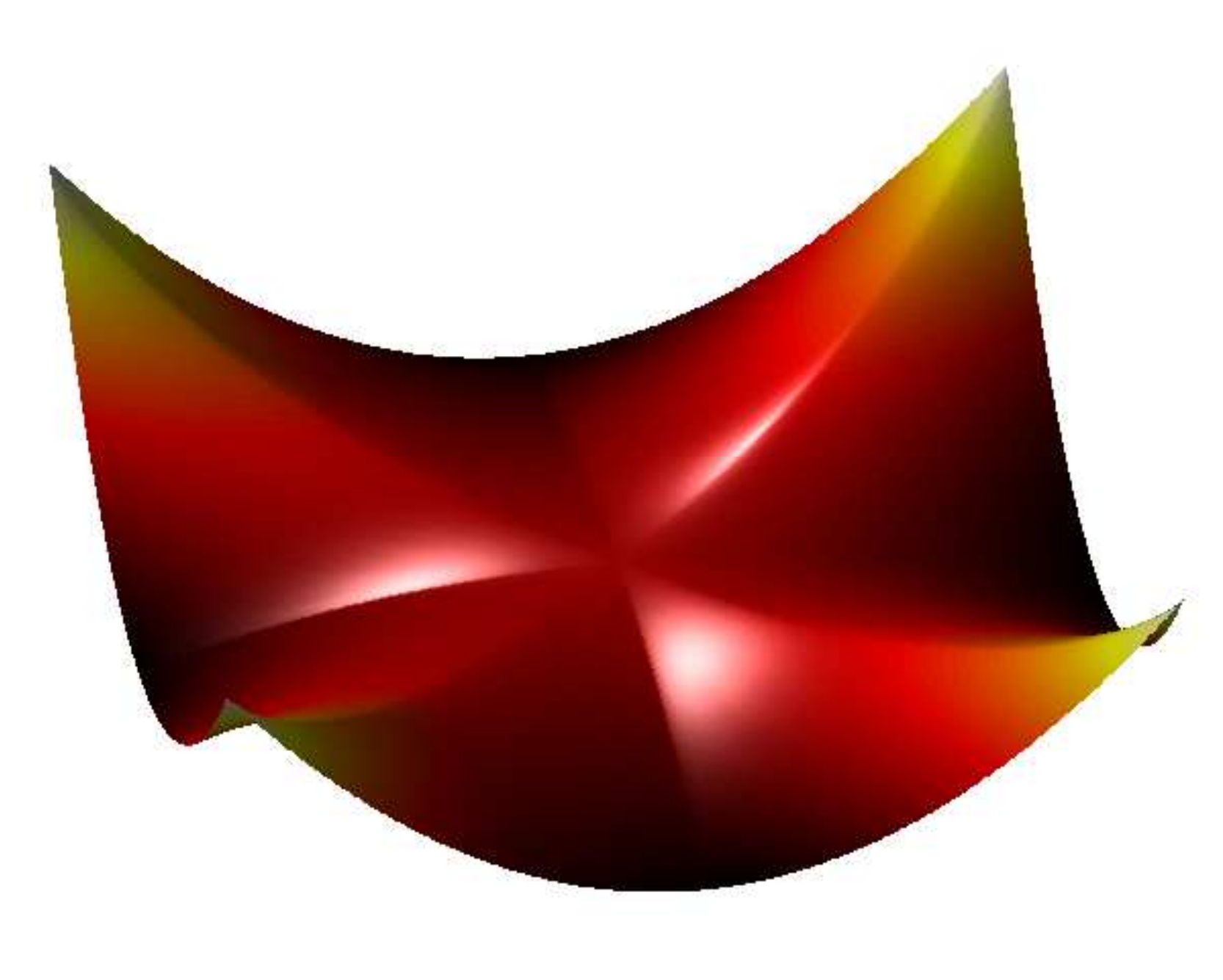}}& 
\resizebox{2.1in}{!}{\includegraphics{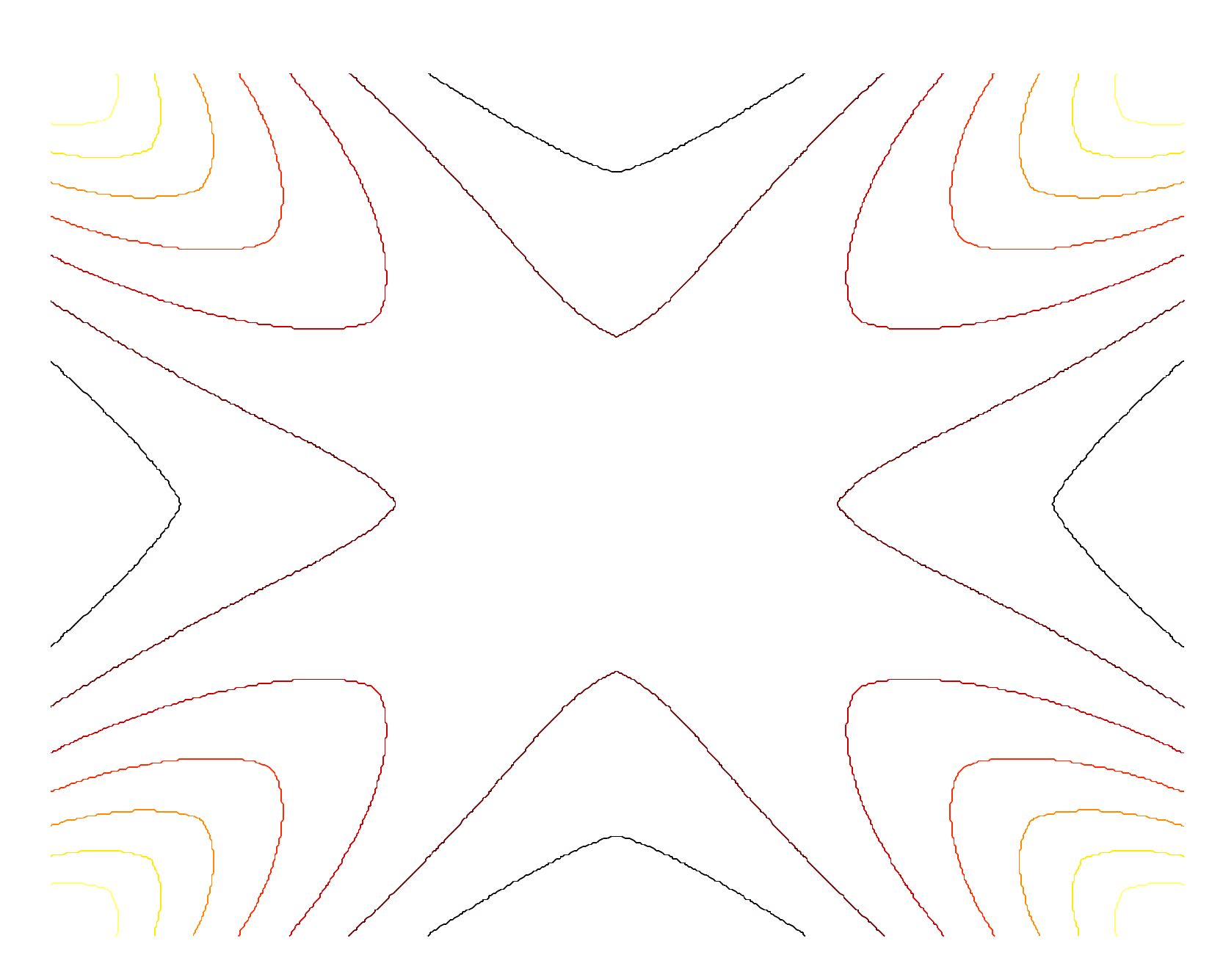}}\\
\resizebox{1.7in}{!}{\includegraphics{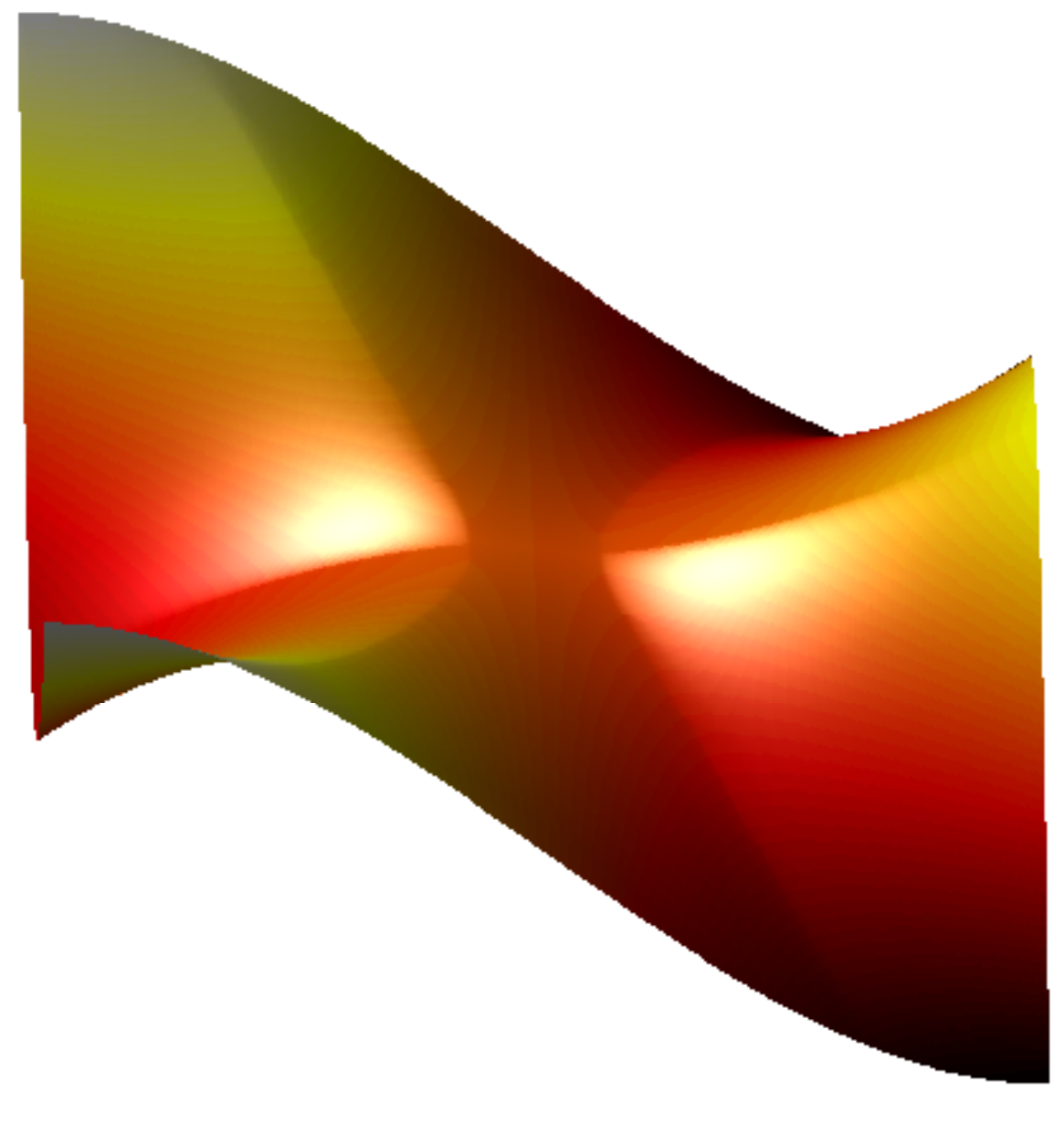}}&
\resizebox{2.1in}{!}{\includegraphics{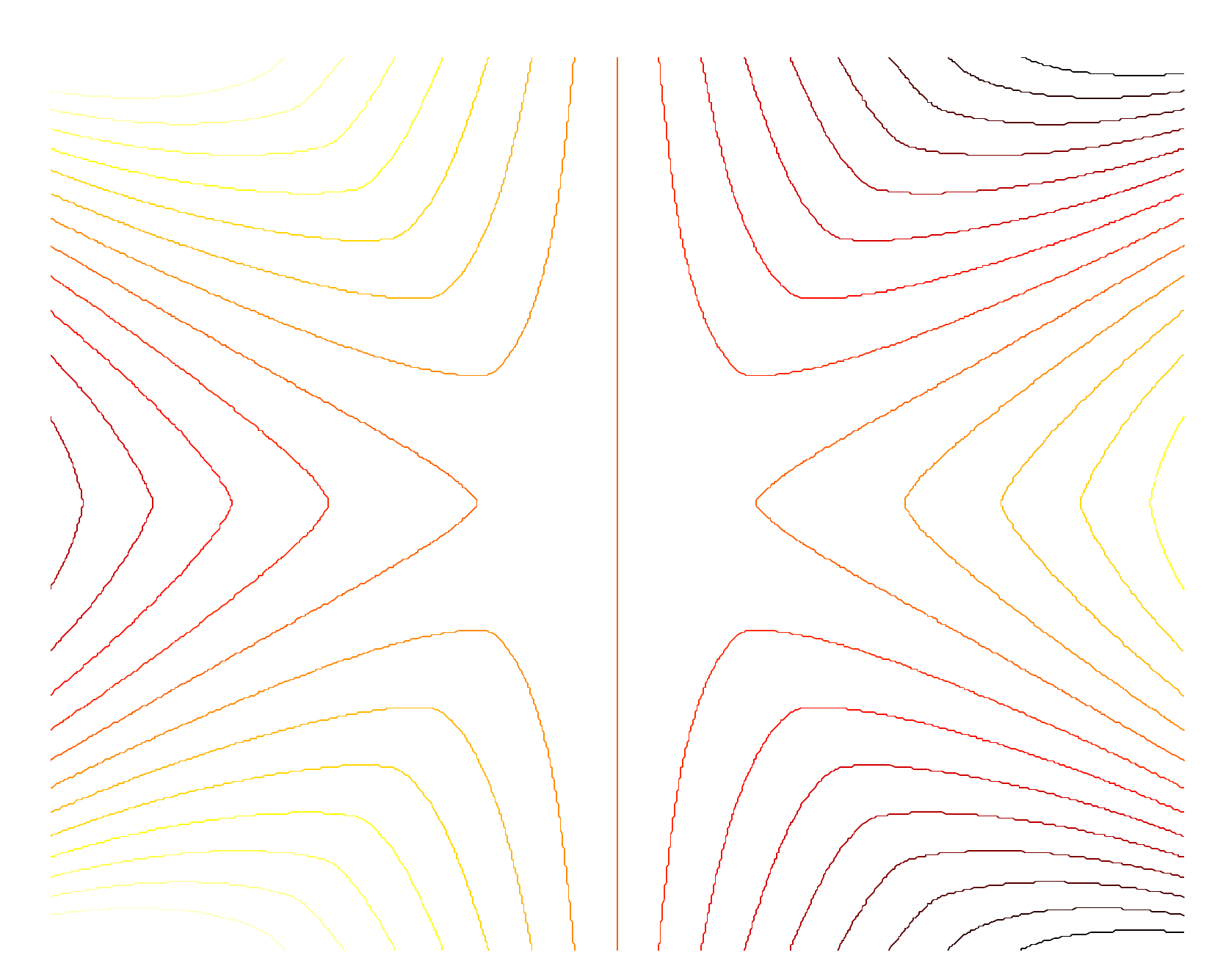}}\\
\resizebox{3in}{!}{\includegraphics{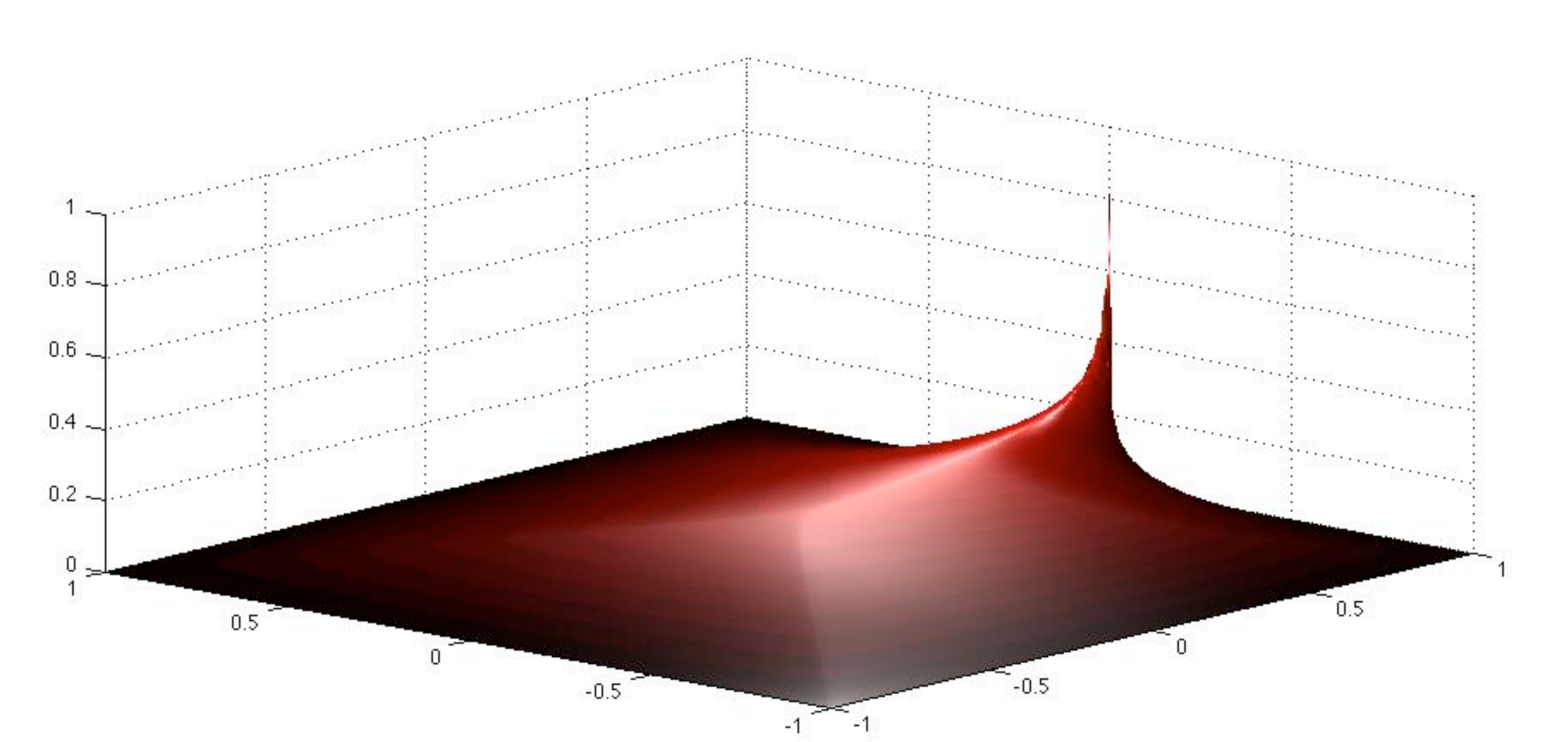}}&
\resizebox{1.7in}{!}{\includegraphics{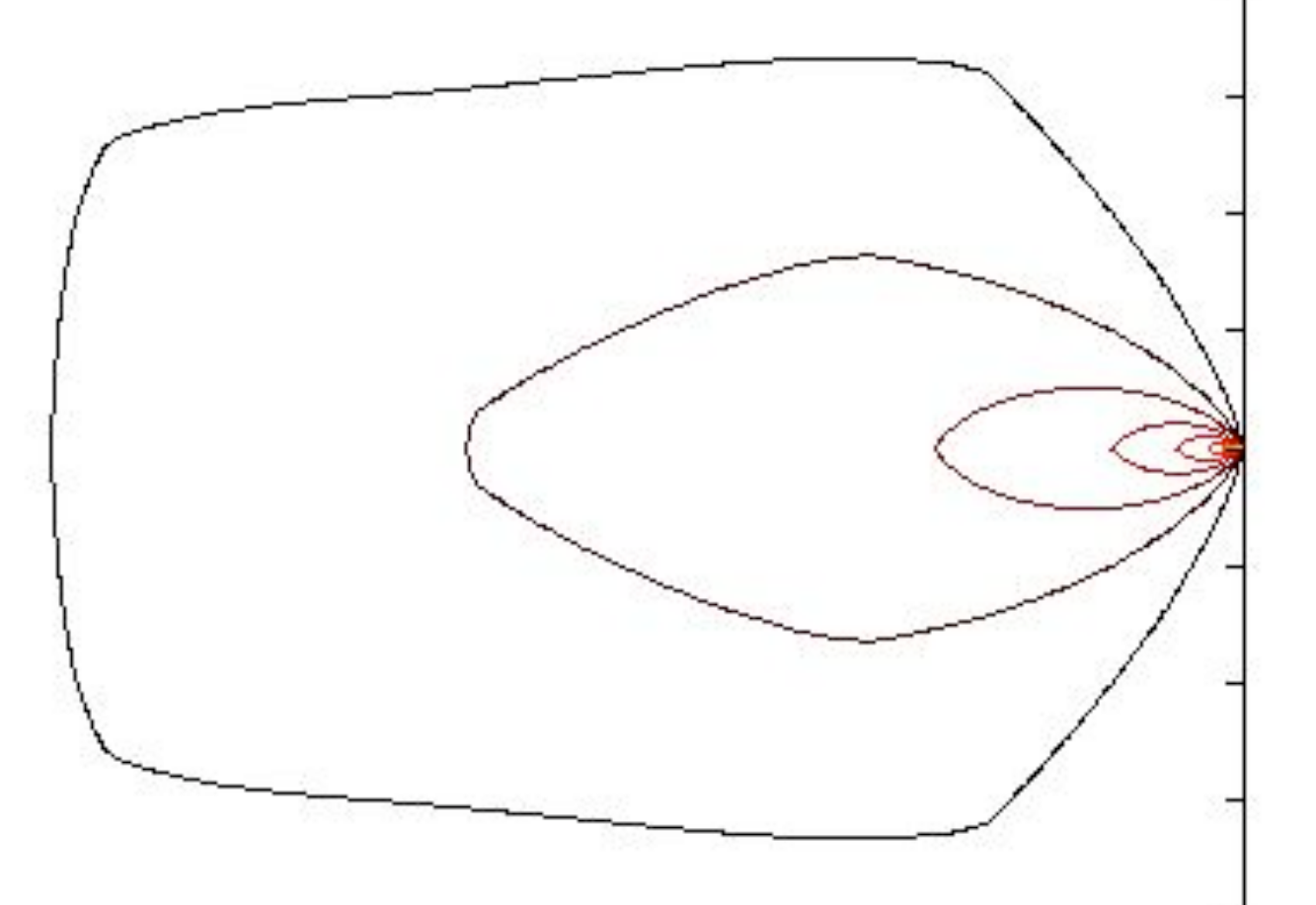}}
\end{tabular}
\end{center}
\caption{Test 1, numerical solutions and corresponding contour plots when $\Omega=(-1,1)^2$, $p=\infty$, $f=0$, with boundary data $F(x,y)=|x|^2 \, |y|^2$, $F(x,y)=x^3-3xy^2$, and $F$ the characteristic function of  point $(1,0)$. }
\end{figure} 

\medskip{\bf Test 2.}  We next consider the problem for the game $p$-Laplacian in a case where we were able to compute the exact solution. Starting with $\Omega=B(0,1)$, $f\equiv1$ and $F\equiv0$, and working in radial coordinates we derived the solution for any $p\geq 2$, that is $v(x,y)=  \frac{1-x^2 -y^2}2$.  This  is an analytic function in the whole plane and has a unique extrema at the origin. By looking at $v$ on the unit square $Q=(-1,1)\times (-1,1)$, we see that this function verifies, for any $p\geq 2$,  $-\Delta_p^G v = 1$ in $Q$. To test our code, we have implemented it on the problem:  $-\Delta_p^G u = 1$ in $Q$, for $F(x,y) =  \frac{1-x^2 -y^2}2$ on the boundary.  We summarize the numerical results obtained with scheme (\ref{scheme-ellip}) for the cases  $p=5$ and $p=\infty$ in Tables~2 and 3 below, showing the $L^\infty$-errors and the number of iterations of the algorithm until convergence ($E_n\le10^{-5}$, for this test) for different combinations of levels and directions. The initial iteration was set to $\min F= -1/2$ at the interior nodes.

%\begin{table}[ht]\label{tab2}
%\small
%\begin{center}
%\begin{tabular}{|c|c|c|c|}
%\hline
%$ {\Delta t}$&Levels&16 directions & 24 directions\\
%\hline
%1/300&2&0.0432 (1150)& 0.0398 (1150)\\
%\hline
%1/300&4&0.01224 (1400)&0.0079 (1100)\\
%\hline
%1/500&2&0.0419 (1550)&0.0390 (1750)\\
%\hline
%1/500&4&0.0117 (1900)&0.0075 (1750)\\
%\hline
%$\alpha^2 \, h^2/2$&2&0.0434 (250)&0.0408 (300)\\
%\hline
%$\alpha^2 \, h^2/2$&4&0.0162 (150)&0.0125 (250)\\
%\hline
%\end{tabular}
%\caption{Test 2, $L^\infty$-errors and iterations for $p=5$, $f=1$, $F=(1-x^2 -y^2)/2$. }
%\end{center}
%\end{table}

\begin{table}[ht]\label{tab2bis}
\small
\begin{center}
\begin{tabular}{|c|c|c|}
\hline
Nodes (levels)& 16 directions error (iter) & 24 directions error (iter)\\
\hline
21 (2) &0.0634 (163)&0.0617 (180) \\
\hline
21 (4) &0.0241 (50)&0.0192 (107) \\
\hline
41 (4) &0.0201 (213) & 0.0191 (163) \\
\hline
\end{tabular}
\caption{Test 2, $L^\infty$-errors and iterations for $p=5$, $f\equiv1$, $F(x,y)=(1-x^2 -y^2)/2$, $\beta=0.9$. }
\end{center}
\end{table}
%\begin{table}[ht]\label{tab2}
%\caption{$p=5$, $f=1$, $F=(1-x^2 -y^2)/2$ on $\partial Q$,  $k=21$. }
%\small
%\begin{center}
%\begin{tabular}{|c|c|c|c|}
%\hline
%$ {\Delta t}$&Levels&16 directions & 24 directions\\
%\hline
%1/300&2&$error=0.0432$& $error=0.0398$\\
%&&$n=1150$& $n=1150$\\
%\hline
%1/300&4&$error=0.01224$&$error=0.0079$\\
%&&$n=1400$&$n=1100$ \\
%\hline
%1/500&2&$error=0.0419$&$error=0.0390$\\
%&&$n=1550$&$n=1750$\\
%\hline
%1/500&4&$error=0.0117$&$error=0.0075$\\
%&&$n=1900$&$n=1750$ \\
%\hline
%$\alpha^2 \, h^2/2$&2&$error=0.0434$&$error=0.0408$\\
%&&$n=250$&$n=300$ \\
%\hline
%$\alpha^2 \, h^2/2$&4&$error=0.0162$&$error=0.0125$\\
%&&$n=150$&$n=250$  \\
%\hline
%\end{tabular}
%\end{center}
%\end{table}

%\begin{table}[ht]\label{tab3}
%\caption{$p=\infty$, $f=1$, $F=(1-x^2 -y^2)/2$ on $\partial Q$, $k=21$. }
%\small
%\begin{center}
%\begin{tabular}{|c|c|c|}
%\hline
%$ {\Delta t}$&Levels&24 directions\\
%\hline
%1/300&2&$error=0.0419, n=1600$ \\
%\hline
%1/300&4&$error=0.0105, n=1450$ \\
%\hline
%1/500&2&$error=0.0419, n=2550$ \\
%\hline
%1/500&4&$error=0.0105, n=2250$ \\
%\hline
%$\alpha^2 \, h^2/2$&2&$error=0.0419, n=300$ \\
%\hline
%$\alpha^2 \, h^2/2$&4&$error=0.0105, n=150$ \\
%\hline
%\end{tabular}
%\end{center}
%\end{table}

\begin{table}[h]\label{tab3bis}
\small
\begin{center}
\begin{tabular}{|c|c|c|}
\hline
Nodes (levels)& 16 directions error (iter) & 24 directions error (iter)\\
\hline
21 (2) &0.0590 (249)&0.0563 (248) \\
\hline
21 (4) &0.0211 (80)&0.0185 (77)\\
\hline
41 (4) &0.0192 (272) & 0.0156 (272) \\
\hline
\end{tabular}
\caption{Test 2, $L^\infty$-errors and iterations for $p=\infty$, $f\equiv1$, $F(x,y)=(1-x^2 -y^2)/2$, $\beta=0.9$. }
\end{center}
\end{table}
%\begin{table}[h]\label{tab3-bis}
%\small
%\begin{center}
%\begin{tabular}{|c|c|c|c|}
%\hline
%$ {\Delta t}$&Levels&directions&$p=5$\\
%\hline
%1/900&4&24&0.0019 (3300)\\
%\hline
%$\alpha^2 \, h^2/2$&4&24&0.0092 (550)\\
%\hline
%\end{tabular}
%\caption{Test 2, $L^\infty$-errors and iterations for $f=1$, $F=(1-x^2 -y^2)/2$, $k=41$. }
%\end{center}
%\end{table}

\medskip{\bf Test 3.}  ({\it The tug-of-war game}) We finally run the code on the rectangle $\Omega=(-2,2)\times(-1,1)$,  for different values of $p$, with $f\equiv1$ and $F\equiv0$. In particular, the case $p=\infty$ and $f\equiv1$ corresponds to a running cost in the tug-of-war game. The exact solution  is not known, and to our knowledge these are the first numerical results for a solution without radial symmetry. More precisely, the explicit solution is known only in a part of the domain, where it can be computed based on ideas from \cite{PSSW}. That is, for any $-1<x<1$, the exact solution is given by  $u(x,y)=\frac{1-y^2}2$. We recover such solution in this part of the domain with our numerics.
 We  set the initial condition ${\bf u}^0$ to $\min F= 0$ in the interior nodes,  and run the scheme $(\ref{scheme-ellip})$ with  {\it 4-level}  circles iterations, with $\beta=0.8$, until the difference of the values at (0,0) of two consecutive  approximations  is less than $10^{-6}$ (note that $u(0,0)=0.5$ is the maximum value of the exact solution). Our tests have shown that the scheme converges for any choice of interior values for ${\bf u}^0$, with of course only a different number of iterations. Here we summarize in Table~4 and in Figure~3 below the results of the simulation with 16 controls (due to the axis-oriented solution, there is no real advantage in this case to use more directions).
 
We have also included Table 5 in order to compare the multi-level circle approaches: as expected the table shows some acceleration of convergence with larger multi-level circles, since the information from the boundary can reach the interior of the domain quicker for larger circles.
 
\begin{figure}[h]
\begin{center}
\begin{tabular}{cc}
\resizebox{2.8in}{!}{\includegraphics{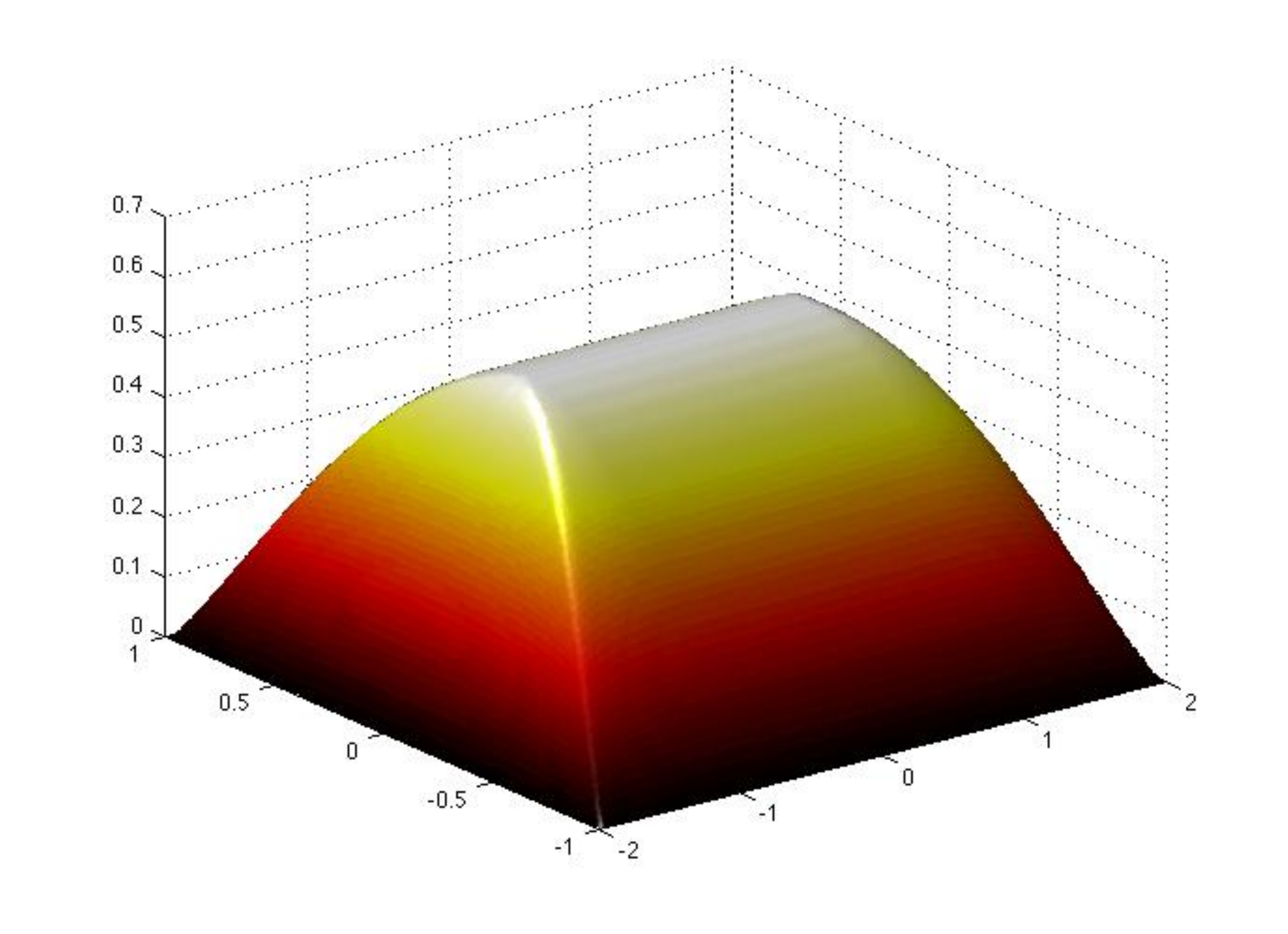}} \ 
\resizebox{2.1in}{!}{\includegraphics{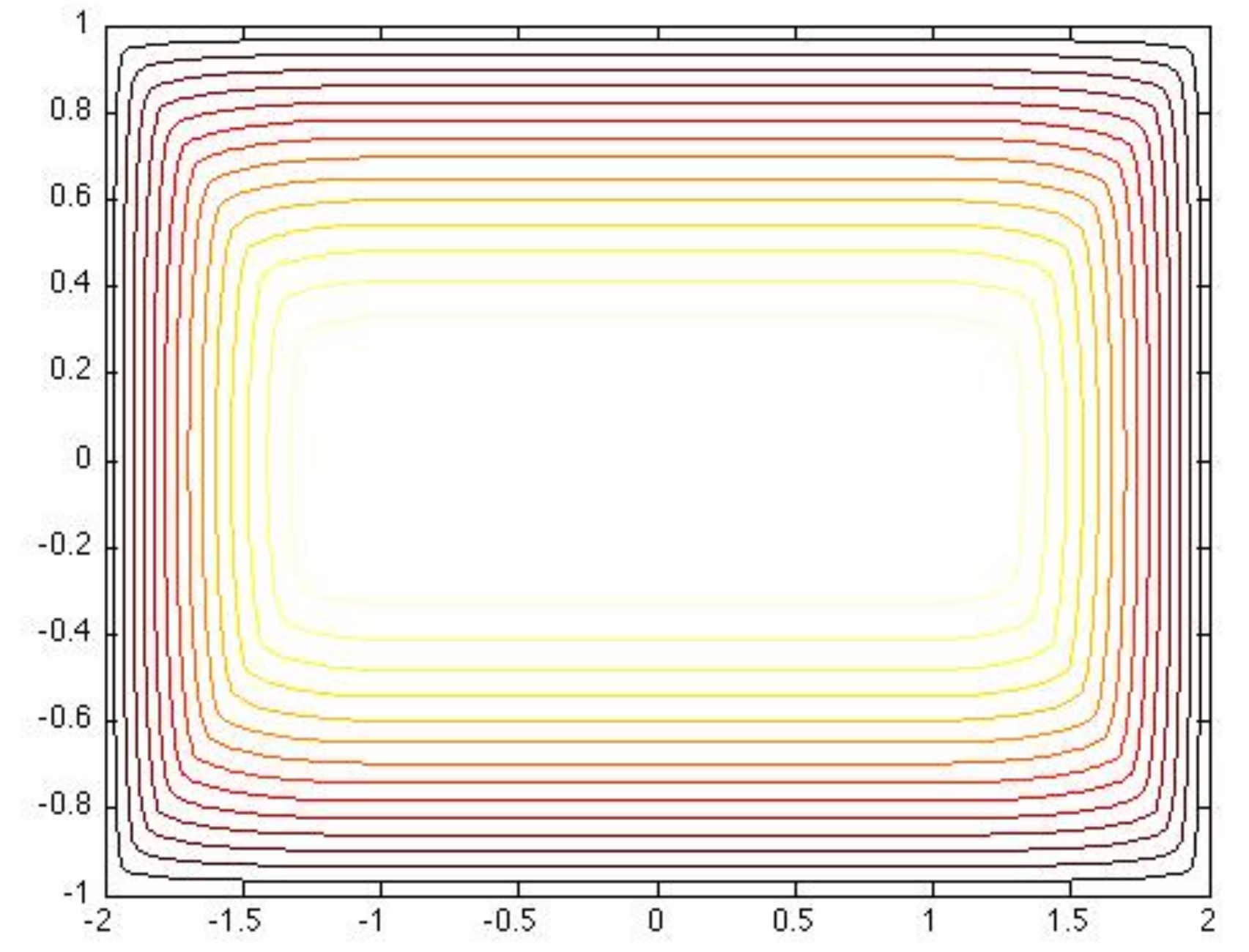}}
\end{tabular}
\end{center}
\caption{Test 3, surface and contour plots of the approximate solution in  the rectangle $(-2,2) \times (-1,1)$ for $p=\infty$, $f\equiv1$, $F\equiv0$, $201\times 101$ grid points and 16 directions.  }
\end{figure}

\begin{table}[h]\label{tab4}
\small
\begin{center}
\begin{tabular}{|c|c|c|c|}
\hline
grid&error at (0,0)&iterations&CPU time\\
\hline
$161\times 81$&0.0276&1112&236\\
\hline
$241\times 121$&0.0155&2205&957\\
\hline
$321\times 161$&0.0094&3578&2681\\
\hline
\end{tabular}
\caption{Test 3, $p=\infty$, $\Omega=(-2,2)\times(-1,1)$, $f\equiv1$, $F\equiv0$, $\beta=0.8$, 16 directions. }
\end{center}
\end{table}

\begin{table}[h]\label{tab5}
\small
\begin{center}
\begin{tabular}{|c|c|c|c|}
\hline
levels&error at (0,0)&iterations&CPU time\\
\hline
4&0.0276&1112&236\\
\hline
2&0.0260&3330&621\\
\hline
1&0.0917&9206&1881\\
\hline
\end{tabular}
\caption{Test 3, $p=\infty$, $\Omega=(-2,2)\times(-1,1)$, $f\equiv1$, $F\equiv0$, $\beta=0.8$, 16 directions, $161\times 81$ nodes. }
\end{center}
\end{table}

\bigskip
%%%%%%%%%%%%%%%%%%%%%%%%%%%%%%%%%%%%%%%%%%%%%%%%%%%%%%%%%%%%%%%%%
%%%%%%%%%%%%%%            APPENDIX
%%%%%%%%%%%%%%%%%%%%%%%%%%%%%%%%%%%%%%%%%%%%%%%%%%%%%%%%%%%%%%%%%

\section[]{Appendix: some $p$-average properties}\label{appe}

For reader's convenience we provide here some elementary properties satisfied by the $p$-average of finite sets of real numbers.

For a fixed set $S=\{s_1,...,s_m\}\subset {\R}$, $s\in {\R}$ and $p>1$,  we define the function
\begin{equation}\label{f-average}
Q(s,S)=\sum_{ j=1}^m |s_j-s|^{p},
\end{equation}
whose derivative with respect to $s$ is easily computed as
\begin{equation}\label{f-der}
\frac{\partial Q}{\partial s}(s,S)=-p \sum_{ j=1}^m  |s_j-s|^{p-1}\,\sgn (s_j-s)=p\,\sum_{ s_j\neq s} |s-s_j|^{p-2}\,(s-s_j).
\end{equation}
For any $p> 1$ one can also compute the second derivative of $Q(s,S)$ with respect to $s$:
\begin{equation}\label{f-der2}
\frac{\partial^2 Q}{\partial s^2}(s,S)=p \, (p-1) \sum_{ j=1}^m  |s_j-s|^{p-2}.
\end{equation}
In the case $1<p<2$ such relation has to be understood in the weak sense of $W^{2,1}$ functions.

\begin{rem}\label{r7.1}
The function $Q$ has exactly one extremum and is convex in $s$, and the p-average $A_p(S)$ is the only value for which $\displaystyle{\frac{\partial Q}{\partial s}(A_p(S),S)=0}$, hence  $\displaystyle{\frac{\partial Q}{\partial s}(s,S)<0}$ for every $s<A_p(S)$, and 
$\displaystyle{\frac{\partial Q}{\partial s}(s,S)>0}$ for $s>A_p(S)$.
Additionally, $A_p(S)$ solves implicitly the following equation:
\begin{eqnarray}\label{chara}
A_p(S)=\frac{\sum_{s_j\neq A_p(S)} |s_j-A_p(S)|^{p-2}\,s_j}{\sum_{ s_j\neq A_p(S)} |s_j-A_p(S)|^{p-2}}\ .
\end{eqnarray}

\end{rem}

\begin{lem}\label{trivial}
 Let $S=\{s_1, s_2,...,s_m\}$ be a finite set of real numbers, and for $k \in {\R}$ let $S+k=\{s_1+k, s_2+k,...,s_m+k\}$. The following assertions hold true for $1\leq p \leq \infty$:
 \begin{equation}\label{transla}
 A_p(S+k)=A_p(S)+k,
 \end{equation}
\begin{equation}
\min_{j=1..m} s_j \leq A_p(S) \leq \max_{j=1..m} s_j  \ .
\end{equation}
\end{lem}

\begin{proof} The cases $p=1$, $p = \infty$ are trivial. For $p> 1$, the first assertion follows by the uniqueness of $A_p(S)$, while the second follows by Remark~\ref{r7.1}, and the fact that if $\displaystyle{s_*=\min_{j=1..m} s_j}$ and $\displaystyle{s^*=\max_{j=1..m} s_j}$ then
by $(\ref{f-der})$ one has $\displaystyle{\frac{\partial Q}{\partial s}(s_* ,S)\leq 0}$ and $\displaystyle{\frac{\partial Q}{\partial s}(s^*,S)\geq 0}$.
\end{proof}

\begin{lem}\label{p-mono}
Let $S=\{s_1, s_2,...,s_m\}$ and $T=\{t_1, t_2,...,t_m\}$ be two finite sets of real numbers having the same number $m$ of elements, and let $1\leq p \leq \infty$ be fixed. If it holds that $t_j\leq s_j$, for every $j=1,...,m$, then we have $A_p(T)\leq A_p(S)$.
\end{lem}

\begin{proof} According to the definition of $p$-average given in $(\ref{p-ave})$ the lemma clearly holds  for the cases $p=1,$ and $ p=\infty$. 

Assume $1<p<\infty$. For $r\in{\cal R}$ define the function $M_p(r)= |r|^{p-2} r$ if  $r\neq 0$ and $M_p(0)=0$, note that $M_p(r)$ is a continuous increasing function.

Given $t:=A_p(T)$, by Remark~\ref{r7.1} we know that $\frac{\partial Q}{\partial s}(t,T)=0$, thus by equation (\ref{f-der}) we obtain 
 $\sum_{ j=1}^m M_p(t-t_j) =0$. But, since $t_j \leq s_j$ it holds $M_p(t-s_j) \leq M_p(t-t_j)$, therefore $\sum_{ j=1}^m M_p(t-s_j) \leq \sum_{ j=1}^m M_p(t-t_j) =0$, which gives  $\frac{\partial Q}{\partial s}(t,S)\leq 0$. By  Remark~\ref{r7.1}, we conclude $A_p(S) \geq t:=A_p(T)$.

\end{proof}

\begin{lem}\label{p-stab}
Let $S$ and $T$ be two finite sets of real numbers having the same number of elements, and let $1\leq p \leq \infty$ be fixed. Assume that $S=\{s_1, s_2,...,s_m\}$ and $T=\{t_1, t_2,...,t_m\}$ verify $t_j= s_j+\delta_j$, for every $j=1,...,m$,   where  $|\delta_j|<\delta$ for some $\delta >0$, then one has 
\begin{equation}\label{Lip}
A_p(S)-\delta \leq A_p(T)\leq A_p(S) + \delta.
\end{equation}
\end{lem}
\begin{proof}
Since $s_j-\delta \leq t_j \leq s_j+\delta$, Lemma~\ref{p-mono} implies $A_p(S-\delta)\leq A_p(T)\leq A_p(S+\delta)$. But, from equation (\ref{transla}) we know
$A_p(S-\delta)=A_p(S)-\delta$, and $A_p(S+\delta)=A_p(S)+\delta$, so that (\ref{Lip}) follows.
 \end{proof}

\bigskip{\bf Acknowledgements.} We wish to thank one of the referees for his/her useful suggestions which contributed to improve our presentation and to simplify some of the proofs.

 \end{document}